\documentclass[12pt]{alt2022} % Include author names

\usepackage{smile}
\usepackage{xcolor}   
\usepackage{colortbl}
\definecolor{LightCyan}{rgb}{0.8, 0.9, 1}
\usepackage{Commands}
\title[Faster Perturbed Stochastic Gradient Methods for Finding Local Minima]{Faster Perturbed Stochastic Gradient Methods for Finding Local Minima}
\usepackage{times}
\altauthor{%
 \Name{Zixiang Chen}\thanks{Equal contribution} \Email{chenzx19@cs.ucla.edu}\\
 \addr Department of Computer Science, University of California, Los Angeles, CA 90095, USA
 \AND
 \Name{Dongruo Zhou}{\color{blue}{$^*$}} \Email{drzhou@cs.ucla.edu}\\
 \addr Department of Computer Science, University of California, Los Angeles, CA 90095, USA%
 \AND
 \Name{Quanquan Gu} \Email{qgu@cs.ucla.edu}\\
 \addr Department of Computer Science, University of California, Los Angeles, CA 90095, USA%
}

\begin{document}

\maketitle

\begin{abstract}
%Escaping from saddle points and finding local minimum is a central problem in nonconvex optimization. In this paper, we propose a novel algorithm named stochastic adaptive recursive gradient descent (SATURN) for finding local minima.  We show that SATURN can find an $O(\epsilon, \epsilon_{H})$-approximate local minima within $\tilde O(\epsilon^{-3} + \epsilon_{H}^{-6})$ stochastic gradient evaluations. Our algorithm is based upon a stochastic recursive gradient and a new adaptive learning rate schedule. To the best of our knowledge, SATURN is the first pure stochastic gradient-based local minima finding algorithm with $\tilde O(\epsilon^{-3}+\epsilon_{H}^{-6})$ gradient complexity, without accessing a negative-curvature search subroutine.

%\dongruo{remember to change the name in abstract}

Escaping from saddle points and finding local minimum is a central problem in nonconvex optimization.
Perturbed gradient methods are perhaps the simplest approach for this problem. However, to find $(\epsilon, \sqrt{\epsilon})$-approximate local minima, the existing best stochastic gradient complexity for this type of algorithms is $\tilde O(\epsilon^{-3.5})$, which is not optimal.
In this paper, we propose \texttt{LENA} (\textbf{L}ast st\textbf{E}p shri\textbf{N}k\textbf{A}ge), a faster perturbed stochastic gradient framework for finding local minima.  We show that $\algname$ with stochastic gradient estimators such as SARAH/SPIDER and STORM can find $(\epsilon, \epsilon_{H})$-approximate local minima within $\tilde O(\epsilon^{-3} + \epsilon_{H}^{-6})$ stochastic gradient evaluations (or $\tilde O(\epsilon^{-3})$ when $\epsilon_H = \sqrt{\epsilon}$). 
The core idea of our framework is a step-size shrinkage scheme to control the average movement of the iterates, which leads to faster convergence to the local minima. %This is new and of independent interest. 

\end{abstract}

% \begin{enumerate}
%     \item Keep $\eh$ to match the lower bound
%     \item Initial point $\xb_1$, big batch size $B$, minibatch size $b$, loop length $q$, perturbation radius $r$, threshold parameters $\totalepoch$, $\gthres$, $\gmove$, $\fdown$, $\ttres$
%     \item Gradient $\eg$, negative curvature $\eh$, probability $\delta$
%     \item $\eh^{-6}$ to $\eh^{-5}$
%     \item b = 1
%     \item storm
% \end{enumerate}

% \begin{itemize}
%     \item $\cA^2$ to $\barD$
% \end{itemize}

\section{Introduction}

In this paper, we focus on the following optimization problem
\begin{align}
    \min_{\xb \in \RR^d} F(\xb) := \EE_{\bxi}[f(\xb; \bxi)],\label{x1}
\end{align}
where $f(\xb; \bxi): \RR^d\rightarrow \RR$ is a stochastic function indexed by some random vector $\bxi$, and it is differentiable and possibly nonconvex. We consider the case where only the stochastic gradients $\nabla f(\xb; \bxi)$ are accessible. \eqref{x1} can unify a variety of stochastic optimization problems, such as finite-sum optimization and online optimization. Since in general, finding global minima of a nonconvex function could be an NP-hard problem \citep{hillar2013most}, one often seeks to finding an $(\eg, \eh)$-approximate local minimum $\xb$, i.e., $\|\nabla F(\xb)\|_{2} \leq \eg$ and $\lambda_{\min}\big(\nabla^{2}F(\xb)\big) \geq -\eh$, where $\nabla F(\xb)$ is the gradient of $F$ and $\lambda_{\min}\big(\nabla^{2}F(\xb)\big)$ is the smallest eigenvalue of the Hessian of $F$ at $\xb$. In many machine learning applications such as matrix sensing and completion \citep{bhojanapalli2016global,ge2017no}, it suffices to find local minima due to the fact that all local minima are global minima.

For the case where $f$ is a deterministic function, it has been shown that vanilla gradient descent fails to find local minima efficiently since the iterates will get stuck at saddle points for exponential time \citep{du2017gradient}. To address this issue, the simplest idea is to add random noises as a perturbation to the stuck iterates.  When second-order smoothness is assumed, \citet{jin2017escape} showed that the simple perturbation step is enough for gradient descent to escape saddle points and find $(\epsilon, \sqrt{\epsilon})$-approximate local minima within $\tilde O(1/\epsilon^2)$ gradient evaluations, which matches the number of gradient evaluations for gradient descent to find $\epsilon$-stationary points. Faster convergence rate can be attained by applying accelerated gradient descent \citep{carmon2018accelerated, jin2018accelerated} or cubic regularization with Hessian-vector product \citep{agarwal2017finding}.

Such results suggest that perturbed gradient methods can find local minima efficiently, at least for deterministic optimization. When it comes to stochastic optimization, a natural question arises:
\begin{center}
\emph{Can perturbed \textbf{stochastic} gradient methods find local minima efficiently?}    
\end{center}

To answer this question, we first look into existing results of perturbed stochastic gradient methods for finding local minima.  \citet{ge2015escaping} showed that perturbed stochastic gradient descent can find $(\epsilon, \sqrt{\epsilon})$-approximate local minima within $\tilde O(\text{poly}(\epsilon^{-1}))$ stochastic gradient evaluations. \citet{daneshmand2018escaping} showed that under a specific CNC condition, stochastic gradient descent is able to find $(\epsilon, \sqrt{\epsilon})$-approximate local minima within $\tilde O (1/\epsilon^5)$ stochastic gradient evaluations. Later on, \citet{li2019ssrgd} showed that simple stochastic recursive gradient descent (SSRGD) can find $(\epsilon, \sqrt{\epsilon})$-approximate local minima within $\tilde O(\epsilon^{-3.5})$ stochastic gradient evaluations, which is the state-of-the-arts to date. However, none of these results by perturbed stochastic gradient methods matches the optimal result $\tilde O(\epsilon^{-3})$ for finding $\epsilon$-stationary points, achieved by stochastic recursive gradient algorithm  (SARAH) \citep{nguyen2017stochastic, pham2020proxsarah}, stochastic path-integrated differential
estimator (SPIDER) \citep{fang2018spider}, stochastic nested variance-reduced gradient (SNVRG) \citep{zhou2020stochastic} and stochastic recursive momentum (STORM) \citep{cutkosky2019momentum} (See also \citet{arjevani2019lower} for the lower bound results). Therefore, whether perturbed stochastic gradient methods can find local minima as efficiently as finding stationary points still remains unknown.

% This kind of setting can be widely found in machine learning problem especially the training of neural networks. Generally finding the global minimum of $F(\xb)$ is NP-hard problem . Therefore, we can only find approximate solution. One solution is first-order stationary point, which only uses.. SPIDER, SARAH, SNVRG, has attained optimal convergence rate. However, first-order has limits, such as ... Therefore, we instead aim at finding a $(\eg, \eh)$ second-order stationary point,  which avoids the algorithm converges to bad saddle point and is a necessary condition for $\xb$ to be a local minimum. The result of SGD finding an $(\eg, \eh)$ approximate approximate local minima can be divided into three categories. 

In this work, we give an affirmative answer to the above question. We propose a general framework named $\algname$, which works together with existing popular stochastic gradient estimators such as SARAH/SPIDER and STORM to find approximate local minima efficiently. 
We summarize our contributions as follows:
\begin{itemize}
\item We prove that $\algname$ finds $(\eg, \eh)$-approximate local minima within $\tilde O(\eg^{-3} + \eh^{-6})$ stochastic gradient evaluations. Specifically, in the classic setting where $\eh = \sqrt{\eg}$, our $\algname$ together with the SARAH/SPIDER estimator enjoys an $\tilde O(\epsilon^{-3})$ stochastic gradient complexity, which outperforms previous best known complexity result $\tilde O(\epsilon^{-3.5})$ achieved by \citet{li2019ssrgd}. Our result also matches the best possible complexity result $\tilde O(\epsilon^{-3})$ achieved by negative curvature search based algorithms \citep{fang2018spider, zhou2020stochastic}, which suggests that \emph{simple} methods such as perturbed stochastic gradient methods can find local minima as efficiently as the more complicated ones. 
\item In addition, we show that $\algname$ with a recent proposed STORM estimator is also able to find $(\eg, \eh)$-approximate local minima within $\tilde O(\eg^{-3} + \eh^{-6})$ stochastic gradient evaluations. 
\item At the core of our $\algname$ is a novel Last step shrinkage scheme to control the average movement of the iterates, which may be of independent interest to other related nonconvex optimization algorithm design. 
\end{itemize}
To compare with previous methods, we summarized related results of stochastic first-order methods for finding local minima in Table \ref{table:comparison}.

%\todoq{combine first and second column, and put the reference right after the algorithm name}

% \begin{table*}[ht] 
% \vspace{-.1in}
% \small
% \caption{Comparison of of different optimization algorithm for find approximate local minima of non convex online problems. 
% }

% \label{table:comparison}

% \begin{center}
% \begin{tabular}{lllll}
% \toprule 
% &Algorithm& Stochastic gradient complexity& Classic Setting & Neon2\\
% \midrule
% \cite{allen2018natasha}& Neon2+Natasha2 & $\tilde{\cO}(\eg^{-3.25} + \eg^{-3}\eh^{-1}+ \eh^{-5})$& $\tilde{\cO}(\eg^{-3.5})$ & needed\\
% \cite{allen2018neon2}&Neon2+SCSG& $\tilde{\cO}(\eg^{-10/3} + \eg^{-2}\eh^{-3}+ \eh^{-5})$&$\tilde{\cO}(\eg^{-3.5})$ &needed\\
% \cite{zhou2018finding}& $\text{SNVRG}^{+}$+Neon2&$\tilde{\cO}(\eg^{-3} + \eg^{-2}\eh^{-3}+ \eh^{-5})$&$\tilde{\cO}(\eg^{-3.5})$ &needed\\
% \cite{fang2018spider}&$\text{SPIDER-SFO}^{+}$(+Neon2)& $\tilde{\cO}(\eg^{-3} + \eg^{-2}\eh^{-2}+ \eh^{-5})$&$\tilde{\cO}(\eg^{-3})$ &needed\\
% \midrule
% \cite{ge2015escaping}&Perturbed SGD&$\text{Poly}(d, \eg^{-1}, \eh^{-1})$&$\tilde{\cO}(\text{Poly}(\eg^{-1}))$ &No\\
% \cite{daneshmand2018escaping}&CNC-SGD& $\tilde{\cO}(\eg^{-4} + \eh^{-10})$&$\tilde{\cO}(\eg^{-5})$ &No\\
% \citet{li2019ssrgd}&SSRGD&$\tilde{\cO}(\epsilon^{-3}+ \epsilon^{-2}\epsilon_{H}^{-3} + \epsilon^{-1}\epsilon_{H}^{-4})$&$\tilde{\cO}(\eg^{-3.5})$&No\\
% \rowcolor{LightCyan}\textbf{This paper}&SATURN&$\tilde{\cO}(\epsilon^{-3}+\epsilon_{H}^{-6})$&$\tilde{\cO}(\eg^{-3})$ 
%  &No\\
% \bottomrule
% \end{tabular}
% \end{center}
% \vspace{-.1in}
% \end{table*}

\begin{table*}[ht] 
\vspace{-.1in}
\small
\caption{Comparison of different optimization algorithm for find approximate local minima of non convex online problems. 
}

\label{table:comparison}

\begin{center}
\begin{tabular}{llll}
\toprule 
Algorithm& Gradient complexity& Classic Setting & Neon2\\
\midrule
 Neon2+Natasha2 \citep{allen2018natasha} & $\tilde{\cO}(\eg^{-3.25} + \eg^{-3}\eh^{-1}+ \eh^{-5})$& $\tilde{\cO}(\eg^{-3.5})$ & needed\\
Neon2+SCSG \citep{allen2018neon2}& $\tilde{\cO}(\eg^{-10/3} + \eg^{-2}\eh^{-3}+ \eh^{-5})$&$\tilde{\cO}(\eg^{-3.5})$ &needed\\
 $\text{SNVRG}^{+}$+Neon2 \citep{zhou2020stochastic}&$\tilde{\cO}(\eg^{-3} + \eg^{-2}\eh^{-3}+ \eh^{-5})$&$\tilde{\cO}(\eg^{-3.5})$ &needed\\
$\text{SPIDER-SFO}^{+}$(+Neon2)\citep{fang2018spider}& $\tilde{\cO}(\eg^{-3} + \eg^{-2}\eh^{-2}+ \eh^{-5})$&$\tilde{\cO}(\eg^{-3})$ &needed\\
\midrule
Perturbed SGD \citep{ge2015escaping}&$\text{Poly}(d, \eg^{-1}, \eh^{-1})$&$\tilde{\cO}(\text{Poly}(\eg^{-1}))$ &No\\
CNC-SGD \citep{daneshmand2018escaping}& $\tilde{\cO}(\eg^{-4} + \eh^{-10})$&$\tilde{\cO}(\eg^{-5})$ &No\\
SSRGD\footnotemark[1]
% \footnote{The complexity result of SSRGD in Table 1 is achieved only when $\epsilon_H>=\epsilon^{1/2}$.}
\citep{li2019ssrgd}&$\tilde{\cO}(\epsilon^{-3}+ \epsilon^{-2}\epsilon_{H}^{-3} + \epsilon^{-1}\epsilon_{H}^{-4})$&$\tilde{\cO}(\eg^{-3.5})$&No\\
\rowcolor{LightCyan}$\algname$ (\textbf{This paper})&$\tilde{\cO}(\epsilon^{-3}+\epsilon_{H}^{-6})$&$\tilde{\cO}(\eg^{-3})$ 
 &No\\
\bottomrule
\end{tabular}
\end{center}
\vspace{-.1in}
\end{table*}

\noindent \textbf{Notation} We use lower case letters to denote scalars,  lower and upper case bold letters to denote vectors and matrices. We use $\| \cdot \|$ to indicate Euclidean norm. We use $\mathbb{B}_{\xb}(r)$ to denote a Euclidean ball center at $\xb$ with radius $r$.We also use the standard $O$ and $\Omega$ notations. We use $\lambda_{\min}(\Mb)$ to denote the minimum eigenvalue of matrix $\Mb$. We say $a_n = O(b_n)$ if and only if $\exists C > 0, N > 0, \forall n > N, a_n \le C b_n$; $a_n = \Omega(b_n)$ if and only if  $\exists C > 0, N > 0, \forall n > N, a_n \ge C b_n$. The notation $\tilde{O}$ is used to hide logarithmic factors.

\section{Related Work}
In this section, we review some important related works. 

\footnotetext[1]{The complexity result of SSRGD in Table \ref{table:comparison} is achieved only when $\epsilon_H\geq\epsilon^{1/2}$.}

\noindent\textbf{Variance reduction methods for finding stationary points.} 
Our algorithm takes stochastic gradient estimators as its subroutine. Specifically, \citet{johnson2013accelerating,xiao2014proximal} proposed stochastic variance reduced gradient (SVRG) for convex optimization in the finite-sum setting. \citet{reddi2016stochastic, allen2016variance} analyzed SVRG for nonconvex optimization. \citet{lei2017nonconvex} proposed a new variance reduction algorithm, dubbed stochastically controlled stochastic gradient (SCSG) algorithm, which finds a $\epsilon$-stationary point within $O(\epsilon^{-10/3})$ stochastic gradient evaluations.
\citet{nguyen2017sarah} proposed SARAH which uses a recursive gradient estimator for convex optimization, and it was extended to nonconvex optimization in \citep{nguyen2017stochastic}. \citet{fang2018spider} proposed a SPIDER algorithm with a recursive gradient estimator and proved an $O(\epsilon^{-3})$ stochastic gradient evaluations to find a $\epsilon$-stationary point, which matches a corresponding lower bound. Concurrently, \citet{zhou2020stochastic} proposed an SNVRG algorithm with a nested gradient estimator and proved an $\tilde O(\epsilon^{-3})$ stochastic gradient evaluations to find a $\epsilon$-stationary point. \citet{wang2019spiderboost} proposed a Spiderboost algorithm with a constant step size, achieves the same $O(\epsilon^{-3})$ gradient complexity. \citet{pham2020proxsarah} extended SARAH \citep{nguyen2017stochastic} to proximal optimization and proved $O(\epsilon^{-3})$ gradient complexity for finding stationary points. Recently, \citet{cutkosky2019momentum} proposed a recursive momentum-based algorithm called STORM and proved an $\tilde O(\epsilon^{-3})$ gradient complexity to find $\epsilon$-stationary points. \citet{tran2019hybrid} proposed a SARAH-SGD algorithm which hybrids both SGD and SARAH algorithm with an $\tilde O(\epsilon^{-3})$ gradient complexity when $\epsilon$ is small. \citet{li2020page} proposed a PAGE algorithm with probabilistic gradient estimator which also attains an $\tilde O(\epsilon^{-3})$ gradient complexity. In our work, we employ SARAH/SPIDER and STORM as the gradient estimator in our $\algname$ framework since they are most representative and simple to use.

\noindent\textbf{Utilizing negative curvature descent to escape from saddle points.}
To escape saddle points, a widely used approach is to first compute the direction of the negative curvature of the saddle point and move away along that direction. In deterministic optimization, \citet{agarwal2017finding} proposed a Hessian-vector product based cubic regularization method which finds $(\eg, \sqrt{\eg})$-approximate local minima within $\tilde O(\eg^{-7/4})$ gradient and Hessian-vector product evaluations. \citet{carmon2018accelerated} proposed an accelerated gradient method with negative curvature finding that finds $(\eg, \sqrt{\eg})$-approximate local minima with in $\tilde O(\eg^{-7/4})$ gradient and Hessian-vector product evaluations. Both of these complexities match the complexity $\tilde O(\eg^{-7/4})$ achieved by perturbed accelerated gradient descent, proposed by \citet{jin2018accelerated}. In stochastic optimization, to find $(\eg, \eh)$-approximate local minima, \cite{allen2018natasha} proposed a Natasha algorithm using Hessian-vector product to compute the negative curvature direction with the total computation cost of $\tilde O(\eg^{-3.25} + \eg^{-3}\eh^{-1}+ \eh^{-5})$. Later, \citet{xu2017first} proposed a Neon method which computes the negative curvature direction with perturbed stochastic gradients, whose total computational cost is $\tilde O(\eg^{-10/3} + \eg^{-2}\eh^{-3}+ \eh^{-6})$. \cite{allen2018neon2} proposed a Neon2 negative curvature computation subroutine with $\tilde O(\eg^{-10/3} + \eg^{-2}\eh^{-3}+ \eh^{-5})$ stochastic gradient evaluations. \citet{fang2018spider} then showed that SPIDER equipped with Neon2 can find $(\eg, \eh)$-approximate local minima within  $\tilde{\cO}(\eg^{-3} + \eg^{-2}\eh^{-2}+ \eh^{-5})$ stochastic gradient evaluations, while independently \citet{zhou2020stochastic} proved that SNVRG equipped with Neon2  can find $(\eg, \eh)$-approximate local minima within  $\tilde{\cO}(\eg^{-3} + \eg^{-2}\eh^{-3}+ \eh^{-5})$ stochastic gradient evaluations. In contrast to this line of works, our algorithm is simpler since it does not need to use the negative curvature search routine.

\section{Preliminaries}
In this section, we present assumptions and definitions that will be used throughout our analysis. We first introduce the standard smoothness and Hessian Lipschitz assumptions.

% \begin{remark}
% This assumption shows that the noise introduced by sampling is bounded, some basic common assumption, such as gradient bound $\|\nabla f\| \le G'$ in Assumptions make in~\citet{cutkosky2019momentum} turns in $G = 2G'$ and $\sigma \le G'$ in our assumption.
% \end{remark}
\begin{assumption}\label{asm:lip}
For all $\bxi$, $f(\cdot; \bxi)$ is $L$-smooth and its Hessian is $\rho$-Lipschitz continuous w.r.t. $\xb$, i.e., for any $\xb_1, \xb_2$, we have that
\begin{align*}
    \|\nabla f(\xb_1; \bxi) - \nabla f(\xb_2; \bxi)\|_2 \le L\|\xb_1 - \xb_2\|_2,\ \|\nabla^2 f(\xb_1; \bxi) - \nabla^2 f(\xb_2; \bxi)\|_2 \le \rho\|\xb_1 - \xb_2\|_2
\end{align*}
\end{assumption}
This assumption directly implies that the expected objective function $F(\xb)$ is also $L$-smooth and its Hessian is $\rho$-Lipschitz continuous. This assumption is standard for finding approximate local minima in all the results presented in Table~\ref{table:comparison}.

% \begin{remark}
% This assumption trivially shows that $F(x)$ is also  $L$-smooth
% since 
% \begin{align*}
%     \|\EE_{\bxi}\nabla f(\xb_1; \bxi) - \EE_{\bxi}\nabla f(\xb_2; \bxi)\|
%     &= \|\EE_{\bxi}(\nabla f(\xb_1; \bxi) - \nabla f(\xb_2; \bxi))\| \\
%     &\le \EE_{\bxi}\|\nabla f(\xb_1; \bxi) - \nabla f(\xb_2; \bxi)\| \\
%     &\le L\|\xb_1 - \xb_2\|_2,
% \end{align*}
% where the first inequality is from Jensen's inequality and $|\cdot|$ is a convex function. Also, assuming $\nabla f$ is $L$-Lip also shows that $f, F$ are both $L$-smooth function, i.e. for any $\xb_1, \xb_2$
% \begin{align*}
%     F(\xb_1) &\le F(\xb_2) + \la \nabla F(\xb_2), \xb_1 - \xb_2\ra + \frac{L}{2}\|\xb_1 - \xb_2\|_2^2\\
%     f(\xb_1; \bxi) &\le f(\xb_2 ; \bxi) + \la \nabla f(\xb_2 ; \bxi), \xb_1 - \xb_2\ra + \frac{L}{2}\|\xb_1 - \xb_2\|_2^2,\quad \forall \bxi
% \end{align*}
% \end{remark}

\begin{assumption}\label{asm:bdvariance}%\todoq{is it variance?}
The squared difference between the stochastic gradient and full gradient is bounded  by $\sigma^{2} < \infty$,  i.e., for any $\xb, \bxi\in \RR^d$, $\|\nabla f(\xb;\bxi) - \nabla F(\xb)\|_{2}^{2} \leq \sigma^{2}$.
\end{assumption}
Assumption \ref{asm:bdvariance} is standard in online/stochastic optimization, including for finding second-order stationary points \citep{fang2018spider, li2019ssrgd}, and immediately implies that the variance of the stochastic gradient is bounded by $\sigma^2$. It can be relaxed to be $\|\nabla f(\xb;\bxi) - \nabla F(\xb)\|_{2}$ has a $\sigma$-Sub-Gaussian tail.

Let $\xb_0\in \RR^{d}$ be the starting point of the algorithm. We assume the gap between the initial function value and the optimal value is bounded.
\begin{assumption}\label{asm:bdvalue}
We have $\Delta = F(\xb_0) - \inf_\xb F(\xb) < +\infty$. 
\end{assumption}
Next, we give the formal definition of approximate local minima (a.k.a., second-order stationary points).
\begin{definition}
We call $\xb \in \RR^{d}$ an $(\eg, \eh)$-approximate local minimum, if
\begin{align*}
\|\nabla F(\xb)\|_{2} \leq \eg, \lambda_{\min}\big(\nabla^{2}F(\xb)\big) \geq -\eh.
\end{align*}
\end{definition}
The definition of $(\eg, \eh)$-approximate local minima is a generalization of the classical $(\eg, \sqrt{\eg})$-approximate local minima studied by \citet{nesterov2006cubic,jin2017escape}.

% \begin{definition}
% We define a general form of recursive gradient estimator as $\db(\xb_0, \dots, \xb_t)$. A recursive gradient estimator is called $(G_1, G_2, H_1, H_2, H_3, (a_i)_i)$-good, if it satisfies the following properties:
% \begin{itemize}
%     \item It satisfies the smoothness property, that is with probability at least $1-\delta$, 
%     \begin{align}
%         \|\db(\xb_0,\dots, \xb_t) - \nabla F(\xb_t)\|_2^2 \leq G_1 + G_2 \sum_{i=0}^{t-1} a_{t-i}\|\xb_{i+1} - \xb_i\|_2^2
%     \end{align}
%     \item It satisfies stabability
%     second-order difference bounded property. Suppose there are two constructed estimators based on two intersecting but different sequences $\xb_0,\dots, \xb_t$ and $\xb_0',\dots, \xb_t'$, where $\xb_i = \xb'_i$ for $0 \leq i \leq m$. Suppose $\|\xb_i - \xb_m\|_2, \|\xb_i' - \xb_m\|_2\leq r$, $m < i  \leq t$, then with probability at least $1-\delta$, 
%     \begin{align}
%         &\|\db(\xb_0,\dots, \xb_t) - \nabla F(\xb_t) - \db(\xb'_0,\dots, \xb'_t) + \nabla F(\xb'_t)\|_2 \notag \\
%         &\leq H_1 + H_2 \max_{m \leq i <t}\|\xb_{i+1} - \xb_i - \xb'_{i+1} + \xb'_i\|_2 + H_3 \max_{m \leq i <t}\|\xb'_{i} - \xb_i\|_2
%     \end{align}
% \end{itemize}
% \end{definition}

\section{The $\algname$ Framework}

In this section, we present our main algorithm $\algname$. We begin with reviewing the mechanism of perturbed gradient descent in deterministic optimization, and then we discuss the main difficulty of extending it to the stochastic optimization case. Finally, we show how we overcome such a difficulty by presenting our $\algname$ framework. 

\noindent\textbf{How does perturbed gradient descent escape from saddle points?} 
We review the perturbed gradient descent \citep{jin2017escape} (PGD for short) with its proof roadmap, which shows how PGD finds $(\epsilon, \sqrt{\epsilon})$-approximate local minima efficiently. In general, the whole process of perturbed gradient descent can be decomposed into several epochs, and each epoch consists of two non-overlapping phases: the \emph{gradient descent phase} ($\gdphase$ for short) and the \emph{Escape from saddle point phase} ($\ncphase$ for short). In each epoch, PGD starts with the $\gdphase$ by default. In the $\gdphase$, PGD performs vanilla gradient descent to update its iterate, until at some iterate $\tilde \xb$, the norm of the gradient $\|\nabla F(\tilde\xb)\|_2$ is less than the target accuracy $\tilde O(\epsilon)$. Then PGD switches to the $\ncphase$. In the $\ncphase$, PGD first adds a uniform random noise (or Gaussian noise) to the current iterate $\tilde\xb$, then it runs $\ttres = \tilde O(\epsilon^{-1/2})$ steps of vanilla gradient descent. PGD then compares the function value gap between the current iterate and the beginning iterate of $\ncphase$ $\tilde\xb$. If the gap is less than a threshold $\cF = \tilde O(\epsilon^{1.5})$, then PGD outputs $\tilde\xb$ as the targeted local minimum. Otherwise, PGD starts a new epoch and performs gradient descent again. 

To see why PGD can find $(\epsilon, \sqrt{\epsilon})$-approximate local minima within $\tilde O(\epsilon^{-2})$ gradient evaluations, we do the following calculation. First, when PGD is in the $\gdphase$, the function value decreases $\tilde O(\epsilon^2)$ per step (following the standard gradient descent analysis). When PGD is in the $\ncphase$, the function value decreases $\cF/\ttres = \tilde O(\epsilon^2)$ per step on average. Therefore, the total number of steps will be bounded by $\tilde O(\epsilon^{-2})$, which is of the same order as GD for finding $\epsilon$-stationary points. 

\noindent\textbf{Limitation of existing methods.}
However, extending the two-phase PGD algorithm from deterministic optimization to stochastic optimization with a competitive gradient complexity is very challenging. We take the SSRGD algorithm proposed by \citet{li2019ssrgd} as an example, which uses SARAH/SPIDER \citep{fang2018spider} as its gradient estimator. Unlike deterministic optimization where we can access the exact function value $F(\xb)$ and gradient $\nabla F(\xb)$ defined in \eqref{x1}, in the stochastic optimization case we can only access the stochastic function $f(\xb; \bxi)$ and the stochastic gradient $\nabla f(\xb; \bxi)$. Therefore, in order to estimate the gradient norm $\|\nabla F(\xb)\|_2$ (which is required at the beginning of $\ncphase$), a naive approach (adapted by \citet{li2019ssrgd}) is to sample a big batch of stochastic gradients $\nabla f(\xb; \bxi_1), \dots, \nabla f(\xb; \bxi_B)$ and uses their mean to approximate $\nabla F(\xb)$. Standard concentration analysis suggests that in order to achieve an $\epsilon$-accuracy, the batch size $B$ should be in the order $\tilde O(\epsilon^{-2})$. Thus, each $\ncphase$ leads to a $\cF = \tilde O(\epsilon^{1.5})$ function value decrease with at least $\tilde O(\epsilon^{-2})$ number of stochastic gradient evaluations, which contributes $\tilde O(1/\epsilon^{1.5}\cdot \epsilon^{-2}) = \tilde O(\epsilon^{-3.5})$ gradient complexity in the end. This is already worse than the $O(\epsilon^{-3})$ gradient complexity of SPIDER for finding $\epsilon$-stationary points. 

\noindent\textbf{Our approach.} Here we propose our $\algname$ framework in Algorithm \ref{alg:Pullback}, which overcomes the aforementioned limitation. %The details of $\algname$ are presented. 
In detail, $\algname$ inherits the two-phase structure of PGD and SSRGD, and it takes either SARAH/SPIDER or STORM \citep{cutkosky2019momentum} as its gradient estimator. The two gradient estimators are summarized as subroutines $\gdspider$ and $\gdstorm$ in Algorithms~\ref{alg:spider} and \ref{alg:storm} respectively, and we use $\db_t$ to denote their estimated gradient at iterate $\xb_t$. The key improvement of $\algname$ is that it directly takes the output of the gradient estimator $\gdname$ to estimate the true gradient $\nabla F(\xb)$, which avoids sampling a big batch of stochastic gradients as in \citet{li2019ssrgd} and thus saves the total gradient complexity. A similar strategy has also been adapted in \citet{fang2018spider}, but for the negative curvature search subroutine. However, such a strategy leads to a new problem to be solved. 

Since we use $\db_t$ to directly estimate $\nabla F(\xb_t)$, in order to make such an estimation valid, we need to guarantee that the error between $\db_t$ and $\nabla F(\xb_t)$ is small enough, e.g., up to $O(\epsilon)$ accuracy. Notice that the recursive structure of SARAH/SPIDER and STORM suggests the following error bound:
\begin{align}
    \forall t,\ \|\db_t - \nabla F(\xb_t)\|_2^2 = \tilde O\bigg(\sum_{i = s_t}^{t-1}\|\xb_{i+1} - \xb_i\|_2^2\bigg), \label{y1}
\end{align}
where $s_t$ is some reference index only related to $t$. 
Therefore, in order to make the error $\|\db_t - \nabla F(\xb_t)\|_2$ small, it suffices to make the movement of the iterates $\|\xb_{i+1} - \xb_i\|_2$ small either individually or on average. In the $\gdphase$, when the norm of the estimated gradient $\|\db_t\|_2$ is large, we use normalized gradient descent, which forces the movement $\|\xb_{i+1} - \xb_i\|_2 = \eta_t \|\db_t\|_{2} = \eta$ to be small. Such an approach is also adapted by \citet{fang2018spider} as a normalized gradient descent for finding either first-order stationary points or local minima. In the $\ncphase$, which starts at the $m_s$-th step.  We achieve this goal by our proposed ``last
step shrinkage'' scheme. In detail, we record the accumulative squared movement starting from $\xb_{m_s+1}$ (after the perturbation step) as $D: = \sum_{i = m_s+1}^{t}\|\xb_{i+1} - \xb_i\|_2^2 = \sum_{i = m_s+1}^{t}\eta_i^2\|\db_i\|_2^2$. When the average movement $D/(t-m_s +1)$ is large, we pull the \emph{last} step size $\eta_t$ back to a smaller value, which forces the average movement $D/(t-m_s +1)$ to be small. Fortunately, such a simple step-size calibration scheme allows us to carefully control the error between $\db_t$ and $\nabla F(\xb_t)$, and to reduce the gradient complexity. %thus completes our $\algname$.

Here we would like to emphasize some features of our algorithm. First, our $\algname$ is a generic framework that can use any stochastic gradient estimator satisfying \eqref{y1}. Second, following previous works \citep{johnson2013accelerating, fang2018spider, cutkosky2019momentum}, our $\algname$ requires the access to a stochastic gradient oracle that simultaneously queries the stochastic gradients at two distinct points with the same randomness $\bxi$. That is stronger than the standard stochastic gradient oracle adapted by SGD and HVP-RVR \citep{arjevani2020second}, which only queries the stochastic gradient at a single point. %We leave to propose a more general algorithm as future work.
%\dongruo{has done the revision}

\begin{algorithm}[t]
\caption{$\algname$}
\begin{algorithmic}[1]\label{alg:Pullback}
\REQUIRE Initial point $\xb_1$, step size $\eta$ and $\eta_H$, perturbation radius $r$, threshold parameter $\ttres$, average movement $\barD$. 
\STATE $\db_1 \leftarrow \gdname(0,\zero, \zero, \xb_1)$, $s\leftarrow 0$, $t\leftarrow1$, \textsc{find}$\leftarrow$false
\WHILE{\textsc{find} = false}
\STATE $s \leftarrow s+1$, $t_s \leftarrow t$, \textsc{find}$\leftarrow$true
\WHILE{$\|\db_{t}\|_2 > \eg$}
\STATE $\eta_t \leftarrow \etag/\|\db_t\|_2$, 
\STATE $\xb_{t+1} \leftarrow \xb_t - \eta_t\db_t$, $\db_{t+1}\leftarrow \gdname(t, \db_t, \xb_t, \xb_{t+1})$, $t \leftarrow t+1$
\ENDWHILE
\STATE $m_s \leftarrow t$, $\bxi\sim \text{Uniform} \ B_{0}(r)$, $\xb_{t+1} \leftarrow \xb_t + \bxi$, $\db_{t+1}\leftarrow \gdname(t, \db_t, \xb_t, \xb_{t+1})$, $t \leftarrow t+1$
\FOR{$k = 0,\dots, \ttres-1$}
\STATE $\eta_t \leftarrow \etah$, $D\leftarrow \sum_{i=m_s}^t \eta_i^2\|\db_i\|_2^2$
\IF {$D>(t-m_s+1)\barD$}
\STATE Set $\eta_t$ such that $\sum_{i=m_s}^t \eta_i^2\|\db_i\|_2^2 = (t-m_s+1)\barD$ \COMMENT{\textcolor{blue}{``Last
Step Shrinkage''}}
\STATE $\xb_{t+1} \leftarrow \xb_t - \eta_t\db_t$, $\db_{t+1}\leftarrow \gdname(t, \db_t, \xb_t, \xb_{t+1})$, $t \leftarrow t+1$, \textsc{find} $\leftarrow$ false, \textbf{break}
\ENDIF
\STATE $\xb_{t+1} \leftarrow \xb_t - \eta_t\db_t$, $\db_{t+1}\leftarrow \gdname(t, \db_t, \xb_t, \xb_{t+1})$, $t \leftarrow t+1$
\ENDFOR
\ENDWHILE
\ENSURE $\xb_{m_s}$
\end{algorithmic}
\end{algorithm}

\begin{algorithm}[t]
\caption{\gdspider($t,\db_t, \xb_t, \xb_{t+1},b,q, B$)}\label{alg:calculate}
\begin{algorithmic}[1]\label{alg:spider}
\REQUIRE Big batch size $B$, mini-batch size $b$, loop length $q$
\IF {$t\ \text{mod}\ q = 0$}
\STATE Generate $\bxi_{t+1}^1,\dots, \bxi_{t+1}^B$. Set $\db_{t+1} \leftarrow \sum_{i = 1}^B \nabla f(\xb_{t+1}; \bxi_{t+1}^i)/B$ 
\ELSE
\STATE Generate $\bxi_{t+1}^1,\dots, \bxi_{t+1}^b$. Set $\db_{t+1}\leftarrow \db_t + \sum_{i = 1}^b \big[\nabla f(\xb_{t+1}; \bxi_{t+1}^i) - \nabla f(\xb_{t}; \bxi_{t+1}^i)\big]/b$
\ENDIF
\ENSURE $\db_{t+1}$
\end{algorithmic}
\end{algorithm}

\section{Main Results}
In this section, we present the main theoretical results. We first present the convergence guarantee of $\algspider$, which uses $\gdspider$ to estimate the gradient $\db_{t}$ in Algorithm~\ref{alg:Pullback}.

\begin{theorem}\label{Theorem Spider} Under Assumptions \ref{asm:lip}, \ref{asm:bdvariance} and \ref{asm:bdvalue}, we choose batch size $B = \tilde O \big(\sigma^{2}\epsilon^{-2} + \sigma^{2}\rho^{2}\eh^{-4}\big)$, $b=q=\sqrt{B}$, set GD phase step size $\etag = \sigma/(2\sqrt{B}L)$, Escape phase step size $\etah = \tilde O(L^{-1})$, perturbation radius $r \leq \min\big\{\sigma/(2\sqrt{B}L), \log(4/\delta)\etah\sigma^{2}/(2B\epsilon), \sqrt{2\log(4/\delta)\etah\sigma^2/(BL)}\big\}$, threshold $\ttres = \tilde O(1/(\etah\eh))$ and $\barD = \sigma^{2}/(4BL^{2})$. Then with high probability,  $\algspider$ can find  $(\epsilon, \epsilon_{H})$-approximate local minima within $\tilde O\big(\sigma L\Delta\epsilon^{-3} + \sigma\rho^{3}L\Delta\eh^{-6}\big)$ stochastic gradient evaluations. 

% \begin{align*}
% \tilde O\bigg(\frac{\sigma L\Delta}{\epsilon^{3}}  + \frac{\sigma\rho^{3}L\Delta}{\epsilon_{H}^{6}}\bigg).  
% \end{align*}
\end{theorem}
\begin{remark}\label{Remark1}
In the classical setting $\epsilon = \sqrt{\epsilon_{H}}$, our result gives $\tilde O(\eg^{-3})$ gradient complexity, which outperforms the best existing result $\tilde O(\eg^{-3.5})$ for perturbed stochastic gradient methods achieved by SSRGD \citep{li2019ssrgd}. 
For sufficiently small $\epsilon$,  \citet{arjevani2020second} proved the lower bound of gradient complexity $\Omega(\epsilon^{-3}+ \epsilon_{H}^{-5})$ for any first-order stochastic methods to find $(\eg, \eh)$-approximate local minima. Our results matches the lower bound $\tilde O(\epsilon^{-3})$ when $\epsilon_H \geq \epsilon^{3/5}$. For the general case, there is still a gap in the dependency of $\epsilon_{H}$ between our result and the lower bound, and we leave to close it as future work. 

% to be specific 
% \begin{align*}
% \Omega\bigg(\frac{\Delta \sigma\sigma_{2}}{\epsilon^{3}}+\frac{\Delta\sigma_{2}^{2}\rho^{2}}{\epsilon_{H}^{5}}\bigg)
% \end{align*}
\end{remark}
%\todoq{change $\Delta$ to $\Delta$}

\begin{algorithm}[t]
\caption{$\gdstorm$($t,\db_t, \xb_t, \xb_{t+1},a, b, B$)}\label{alg:calculate}
\begin{algorithmic}[1]\label{alg:storm}
\REQUIRE Initial batch size $B$, mini batch size $b$ and weight parameter $a$.
\IF {$t = 0$}
\STATE Generate $\bxi_{t+1}^1,\dots, \bxi_{t+1}^{B}$. Set $\db_{t+1} \leftarrow \sum_{i = 1}^{B} \nabla f(\xb_{t+1}; \bxi_{t+1}^i)/B$ 
\ELSE
\STATE Generate $\bxi_{t+1}^1,\dots, \bxi_{t+1}^b$
\STATE Set $\db_{t+1}\leftarrow (1-a)\big[\db_t - \sum_{i = 1}^b \nabla f(\xb_{t}; \bxi_{t+1}^i)/b\big] +  \sum_{i = 1}^b \nabla f(\xb_{t+1}; \bxi_{t+1}^i)/b$
\ENDIF
\ENSURE $\db_{t+1}$
\end{algorithmic}
\end{algorithm}

Next, we present the convergence guarantee of $\algstorm$, which uses the gradient estimator $\gdstorm$ to estimate the gradient $\db_{t}$ in Algorithm \ref{alg:Pullback}.

\begin{theorem}\label{thm:storm} Under Assumptions \ref{asm:lip}, \ref{asm:bdvariance} and \ref{asm:bdvalue}, choose the mini batch size $b = \tilde O \big(\sigma\epsilon^{-1} + \sigma\rho \eh^{-2}\big)$, and initial batch size $B = b^{2}$, set GD phase step size $\etag = \sigma/(2bL)$, Escape phase step size $\etah = \tilde O(L^{-1})$, weight $a = 56^{2}\log(4/\delta)/b$, threshold $\ttres = \tilde O(1/(\etah\eh))$, perturbation radius $r \leq \min\big\{\sigma/(2bL), \log(4/\delta)^{2}\etah\sigma^{2}/(\epsilon b^{2}), \sqrt{2\log(4/\delta)^{2}\etah\sigma^2/(b^{2}L)}\big\}$, and $\barD = \sigma^{2}/(4b^{2}L^{2})$. Then with high probability,  $\algstorm$ can find  $(\epsilon, \epsilon_{H})$-approximate local minima within $\tilde O\big(\sigma L\Delta\epsilon^{-3} + \sigma\rho^{3}L\Delta\eh^{-6}\big)$ stochastic gradient evaluations. 
% \begin{align*}
% \tilde O\bigg(\frac{\sigma L\Delta}{\epsilon^{3}}  + \frac{\sigma\rho^{3}L\Delta}{\epsilon_{H}^{6}}\bigg).  
% \end{align*}
\end{theorem}
\begin{remark}
Different from $\algspider$, the estimation error $\|\db_t - \nabla F(\xb_t)\|_2$ of gradient estimator $\algstorm$ is controlled by the weight parameter $a$. This allows us to come up with a simpler single-loop algorithm instead of a double-loop algorithm. 
\end{remark}

\noindent\textbf{Experiments} We conduct some experiments to validate the practical performance of $\algname$. Due to the space limit, the details of the experiment setup are deferred to Appendix \ref{sec:exp}. Our experiment results suggest that although $\algname$ does not enjoy a strictly better complexity result than existing algorithms such as NEON2-based algorithms \citep{xu2017first, allen2018neon2}, our $\algname$ outperforms them empirically due to its simpler algorithm structure.

% \begin{theorem}[With storm]\label{thm:storm}
% Set the parameters as follows:
% \begin{align}
%     B = ,..., 
% \end{align}
% then with high probability, Algorithm \ref{alg:Pullback} returns an $(10\epsilon, 50\sqrt{\rho\epsilon})$-second order stationary point. The total gradient complexity to run Algorithm \ref{alg:Pullback} is
% \begin{align}
%     O\bigg(\frac{1}{\eg^3} + \frac{1}{\eg\eh^4}\bigg)
% \end{align}
% \end{theorem}

\section{Proof Outline of the Main Results}\label{Sec: Spider-1st}
We outline the proof of Theorem~\ref{Theorem Spider} and leave the proof of Theorem~\ref{thm:storm} to the appendix.

Let $\bepsilon_t$ denote the difference between true gradient $\nabla F(\xb_t)$ and the estimated gradient $\db_t$, which is $\bepsilon_{t}:=\db_t - \nabla F(\xb_t)$. The following lemma suggests that the estimation error $\|\bepsilon_{t}\|_{2}$ can be bounded.
\begin{lemma}\label{lm:boundeps}
Under Assumptions \ref{asm:lip} and \ref{asm:bdvariance}, set $b=q=\sqrt{B}$, $\etag \leq \sigma/(2\sqrt{B}L)$, $r \leq \sigma/(2\sqrt{B}L)$ and $\barD \leq \sigma^{2}/(4BL^{2})$, then with probability at least $1-\delta$, for all $t$ we have 
\begin{align*}
\|\bepsilon_{t}\|_2 \leq \sqrt{8\log(4/\delta)}\sigma/\sqrt{B}.    
\end{align*}
Specifically, by the choice of $B$ in Theorem \ref{Theorem Spider} we have that $\|\bepsilon_{t}\|_{2} \leq \eg/2$.
\end{lemma}

\begin{proof}[Proof of Lemma \ref{lm:boundeps}]
By $\gdspider$ presented in Algorithm \ref{alg:spider} we have
\begin{align}
    &\bepsilon_{t+1} = \frac{1}{B}\sum_{i=1}^B\big[\nabla f(\xb_{t+1}; \bxi_{t+1}^i) - \nabla F(\xb_{t+1})\big],&t\ \text{mod}\ q = 0,\notag \\
    &\bepsilon_{t+1} = \bepsilon_t + \frac{1}{b}\sum_{i=1}^b\big[\nabla f(\xb_{t+1}; \bxi_{t+1}^i) - \nabla f(\xb_{t}; \bxi_{t+1}^i) - \nabla F(\xb_{t+1}) + \nabla F(\xb_t)\big],&t\ \text{mod}\ q \neq 0.\notag
\end{align}
By the $L$-smoothness in Assumption \ref{asm:lip} we have
\begin{align}\nonumber
    \big\|\nabla f(\xb_{t+1}; \bxi_{t+1}^i) - \nabla f(\xb_{t}; \bxi_{t+1}^i) - \nabla F(\xb_{t+1}) + \nabla F(\xb_t)\big\|_2 \leq 2L\|\xb_{t+1} - \xb_t\|_2.
\end{align}
Then by Assumption \ref{asm:bdvariance} and Azuma–Hoeffding inequality (See Lemma \ref{lm:ah-vec} for details), with probability at least $1-\delta$, we have
\begin{align}
    \forall t>0,\ \|\bepsilon_{t+1}\|_2^2 \leq 4\log(4/\delta)\bigg(\frac{\sigma^2}{B} + \frac{4L^2}{b}\sum_{i=\lfloor t/q \rfloor q}^t \|\xb_{i+1} - \xb_{i}\|_2^2\bigg)\label{eq:SpiderD}.
\end{align}
Notice that $\gdspider$ is parallel with $\algname$. Thus we need to further bound \eqref{eq:SpiderD} by considering iterates in three different cases: (1) for step $i$ in the \gdphase, by normalized gradient descent we have $\|\xb_{i+1} - \xb_i\|_2^2 \leq \etag^2$; (2) for $i = m_s$ for some $s$ in the \ncphase, we have $\|\xb_{i+1} - \xb_i\|_2^2 \leq r^{2}$; and (3) for the other steps in \ncphase, we have on average, $\|\xb_{i+1} - \xb_i\|_2^2 \leq \barD$ due to the ``Last
Step Shrinkage" scheme. Therefore we have
\begin{align}\nonumber
    \|\bepsilon_{t+1}\|_2^2 \leq 4\log(4/\delta)\bigg(\frac{\sigma^2}{B} + \frac{4L^2}{b}\cdot q\cdot \max\{\etag^2, r^{2}, \barD\}\bigg) \leq \frac{8\log(4/\delta)\sigma^{2}}{B}.
\end{align}
\end{proof}
Lemma~\ref{lm:boundeps} guarantees that with high probability $\|\nabla F(\xb_{t})\|_{2}\leq \|\db_{t}\|_{2}+\epsilon$, which ensures $\|\nabla F(\xb_{m_{s}})\|_{2}\leq 2\epsilon$ when the algorithm terminates. Next lemma bounds the function value decrease in the \gdphase, which is also valid for \algname-STORM. 

\begin{lemma}\label{lm:gradientdescent}
Suppose the event in Lemma \ref{lm:boundeps} holds, $\eta \leq \epsilon/(2L)$, then for any $s$, we have
\begin{align}
    F(\xb_{t_s}) - F(\xb_{m_{s}}) \geq \frac{(m_{s} - t_{s})\eta\eg}{8}.\notag
\end{align}
The choice of $\eta$ in Theorem \ref{Theorem Spider} further implies that the loss decreases by at least $\sigma \epsilon/(16\sqrt{B}L)$ per step on average.
\end{lemma}

\begin{proof}[Proof of Lemma \ref{lm:gradientdescent}]
For any $ t_s \leq t < m_{s}$, we can show the following property (See Lemma~\ref{lm:basic}),
\begin{align}\label{eq:basic}
F(\xb_{t+1}) \leq F(\xb_t) - \frac{\eta_t}{2}\|\db_t\|_2^2 + \frac{\eta_t}{2}\|\bepsilon_t\|_2^2 + \frac{L}{2}\|\xb_{t+1} - \xb_t\|_2^2.  
\end{align}
Plugging the update rule $\xb_{t+1} = \xb_{t} - \eta_{t}\db_{t}$ into \eqref{eq:basic} gives,
\begin{align*}
F(\xb_{t+1}) &=  F(\xb_t) -\|\xb_{t+1} - \xb_t\|_2^2\bigg(\frac{1}{2\eta_t} - \frac{L}{2}\bigg) +\frac{\eta_t\|\bepsilon_t\|_2^2}{2}\notag \\
& \leq F(\xb_t) -\etag^2\bigg(\frac{1}{2\eta_t} - \frac{L}{2}\bigg) + \frac{\eta_{t}\epsilon^{2}}{8},\notag\\
&\leq F(\xb_t) - \frac{\eta \epsilon}{8}
\end{align*}
where the first inequality holds due to the fact that $\eta_t = \eta/\|\db_t\|_2$ and $\|\bepsilon_{t}\|_{2}\leq \epsilon/2$, and the second inequality is by $\eta_{t} = \eta/\|\db_t\|_2 \leq \eta/\epsilon \leq 1/(2L)$.
\end{proof}

The following lemma shows that if $\xb_{m_{s}}$ is a saddle point, then with high probability, the algorithm will break during the $\ncphase$ and set \textsc{find} to be false. Thus, whenever $\xb_{m_{s}}$ is not a local minima, the algorithm cannot terminate.

% the $\ncphase$ doesn't break then with high probability we find the $(\epsilon, \epsilon_{H})$ stationary point.

% \begin{lemma}\label{lm:hessiandescent}
% Suppose $\barD= \tilde O(L^{2}\etah^{4}\eh^{4}\rho^{-2})$, $\ttres  = \tilde O(\etah^{-1}\eh^{-1}))$, $\etah\leq \tilde O(L^{-1})$, $b \geq q$, $r \leq \tilde O(L\etah\eh\rho^{-1})$. For any $s$, when algorithm does not break, with high probability, $\xb_{m_s}$ satisfies $\lambda_{\min}(\nabla^2 F(\xb_{m_s})) \geq -\eh$.
% \end{lemma}

\begin{lemma}\label{lm:hessiandescent}
Under Assumptions \ref{asm:lip} and \ref{asm:bdvariance}, set perturbation radius  $r \leq L\etah\eh/(C\rho)$,step size $\etah\leq \min\{1/(16L\log(\etah\epsilon_{H}\sqrt{d} LC^{-1}\rho^{-1}\delta^{-1}r^{-1})), 1/(8CL\log \ttres) \} = \tilde O(L^{-1})$, $\ttres  = 2\log(\etah\epsilon_{H}\sqrt{d} LC^{-1}\rho^{-1}\delta^{-1}r^{-1})/(\etah\eh) = \tilde O(\etah^{-1}\eh^{-1})$,  and $\barD < C^{2}L^{2}\etah^{2}\eh^{2}/(\rho^{2}\ttres^{2})$, where $C = O(\log (d\ttres/\delta) = \tilde O(1)$. We also set $b = q =\sqrt{B}\geq 16\log(4/\delta)/(\etah^{2}\eh^{2})$. Then for any $s$, when $\lambda_{\min}(\nabla^2 F(\xb_{m_s})) \leq -\eh$,  with  probability at least $1-2\delta$ algorithm breaks in the $\ncphase$.
\end{lemma}

\begin{proof}[Proof of Lemma~\ref{lm:hessiandescent}] Let $\{\xb_{t}\}, \{\xb_{t}'\}$ be two coupled sequences by running $\algspider$ \ from $\xb_{m_{s}+1}, \xb_{m_{s}+1}'$ with $\xb_{m_{s}+1}-\xb_{m_{s}+1}' = r_{0}\eb_{1}$, where $\xb_{m_{s}+1}, \xb_{m_{s}+1}'\in \mathbb{B}_{\xb_{m_{s}}}(r)$. Here $r_{0} = \delta r/\sqrt{d}$ and $\eb_{1}$ denotes the smallest eigenvector direction of Hessian $\nabla^{2}F(\xb_{m_{s}})$. 

When $\lambda_{\min}(\nabla^2 F(\xb_{m_s})) \leq -\eh$, under the parameter choice in Lemma~\ref{lm:hessiandescent}, we can show that at least one of two sequence will escape the saddle point (See Lemma \ref{lm:hessiandescentbaisc}). To be specific, with probability at least $1-\delta$, 
\begin{align}
    \max_{m_s < t <m_s +\ttres}\{\|\xb_t - \xb_{m_s+1}\|_2, \|\xb_t' - \xb_{m_s+1}'\|_2\} \geq \frac{L\etah\eh}{C\rho}.\label{eq:region}
\end{align} 
\eqref{eq:region} suggests that for any two points $\xb_{m_{s}+1}, \xb_{m_{s}+1}'$ satisfying $\xb_{m_{s}+1}-\xb_{m_{s}+1}' = r_{0}\eb_{1}$, at least one of them will generate a sequence of iterates which finally move more than $L\etah\eh/(C\rho)$. Thus, let $\cS\subseteq \mathbb{B}_{m_{s}}(r)$ be the set of $\xb_{m_{s}+1}$ which will not generate a sequence of iterates moving more than $\frac{L\etah\eh}{C\rho}$, then in the direction $\eb_1$, the "thickness" of $\cS$ is smaller than $r_0$. Simple integration shows that the ratio between the volume of $\cS$ and $\mathbb{B}_{m_{s}}(r)$ is bounded by $\delta$. Therefore, since $\xb_{m_{s}+1}$ is generated from $\xb_{m_{s}}$ by adding a uniform random noise in ball $\mathbb{B}_{m_{s}}(r)$, we conclude that the probability for $\xb_{m_{s}+1}$ locating in $\cS$ is less than $\delta$. 
% implies that in the random perturbation ball $\mathbb{B}_{m_{s}}(r)$, the stuck points can only locate in a short interval (length no greater than $r_{0}$) in the $\eb_{1}$ direction. Due to the choice of $r_{0} = \delta r/\sqrt{d}$, we can show that the probability of $\xb_{m_{s}+1}$ located in this region is less than $\delta$. 
Applying union bound, we get with probability at least $1-2\delta$,
\begin{align}
\exists m_s < t <m_s +\ttres,  \|\xb_t - \xb_{m_s+1}\|_2  \geq \frac{L\etah\eh}{C\rho}. \label{eqref:highpbstationary}
\end{align}
Denote $\mathcal{E}$ as the event that the algorithm does not break in the $\ncphase$. Then under $\mathcal{E}$, for any $m_s < t <m_s +\ttres$, we have
\begin{align}
\|\xb_t - \xb_{m_s+1}\|_2  \leq \sum_{i=m_s+1}^{t-1}\|\xb_{i+1} - \xb_i\|_2 \leq \sqrt{(t-m_s)\sum_{i=m_s}^{t-1} \|\xb_{i+1} - \xb_i\|_2^2} \leq (t-m_s)\sqrt{\barD},\notag
\end{align}
where the first inequality is due to the triangle inequality and the second inequality is due to Cauchy-Schwarz inequality. Thus, by the choice of $\ttres$ and $\barD$, we have
\begin{align*}
    \|\xb_t - \xb_{m_s+1}\|_2 \leq(t-m_{s})\sqrt{\barD} \leq \ttres \sqrt{\barD} <C\cdot\frac{L\etah\eh}{\rho}.
\end{align*}
Then by \eqref{eqref:highpbstationary}, we know that $\mathbb{P}(\cE)\leq 2\delta$. Therefore when $\lambda_{\min}(\nabla^2 F(\xb_{m_s})) \leq -\eh$,  with  probability at least $1-2\delta$, $\algname$ breaks in the $\ncphase$. 
\end{proof}

Next lemma bounds the decreasing value of the function during the $\ncphase$ if the algorithm breaks in the $\ncphase$(i.e. \textsc{find} is false).

\begin{lemma}[localization]\label{lm:localization}
Suppose the result of Lemma \ref{lm:boundeps} holds, set $\etah \leq 1/ \big(L\sqrt{128\log(4/\delta)}\big)$,  $r \leq \min\big\{\log(4/\delta)\etah\sigma^{2}/(2B\epsilon), \sqrt{2\log(4/\delta)\etah\sigma^2/(BL)}\big\}$,  and $\barD = \sigma^{2}/(4BL^{2})$.
Suppose the algorithm breaks in the $\ncphase$ starting at $\xb_{m_s}$, then we have 
% \begin{align}
%     \sum_{i=m_{s}+1}^{t-1}\|\xb_{i+1} - \xb_i\|_2^2 \leq  \etah(F(\xb_{m_{s}+1}) - F(\xb_{t})) + \etah^2 (t-m_s-1)\eg^2.
% \end{align}
% and 
% \begin{align}
%     F(\xb_{m_{s}}) - F(\xb_{t_{s+1}}) \geq (t_{s+1}-m_s)\frac{\sigma^{2}}{32\etah BL^{2}}.
% \end{align}
\begin{align*}
    F(\xb_{m_{s}}) - F(\xb_{t_{s+1}}) \geq (t_{s+1}-m_s)\frac{\log(4/\delta)\etah\sigma^2}{B}.
\end{align*}
\end{lemma}

\begin{proof}[Proof of Lemma \ref{lm:localization}]
For any $m_s < i <t_{s+1}$, we can show the following property (See Lemma~\ref{lm:basic}),
\begin{align}\label{eq:basic2}
F(\xb_{i+1}) \leq F(\xb_i) - \frac{\eta_i}{2}\|\db_i\|_2^2 + \frac{\eta_i}{2}\|\bepsilon_i\|_2^2 + \frac{L}{2}\|\xb_{i+1} - \xb_i\|_2^2.  
\end{align}
Plugging the update rule $\xb_{i+1} = \xb_{i} - \eta_{i}\db_{i}$ into \eqref{eq:basic2} gives,
\begin{align}
F(\xb_{i+1}) 
&\leq F(\xb_i)+ \frac{\eta_i}{2}\|\bepsilon_i\|_2^2 - \bigg(\frac{1}{2\eta_i} - \frac{L}{2}\bigg)\|\xb_{i+1} - \xb_i\|_2^2 \notag \\
&\leq F(\xb_i)+ \frac{\etah}{2}\frac{8\log(4/\delta)\sigma^2}{B} - \frac{1}{4\etah}\|\xb_{i+1} - \xb_i\|_2^2\label{eq:local11}
\end{align}
where the the second inequality holds due to Lemma~\ref{lm:boundeps} and $\eta_i \leq \etah \leq 1/(2L)$ for any $m_s < i < t_{s+1}$. Telescoping $\eqref{eq:local11}$ from $i=m_{s}+1$ to $t_{s+1}-1$, we have
\begin{align*}
F(\xb_{t_{s+1}}) &\leq F(\xb_{m_{s}+1}) + 4\etah \log(4/\delta)(t_{s+1}-m_s-1)\frac{\sigma^2}{B} - \frac{1}{4\etah}\sum_{i=m_{s}+1}^{t_{s+1}-1}\|\xb_{i+1} - \xb_i\|_2^2.
\end{align*}
Finally, we have
\begin{align}
    F(\xb_{m_{s}+1}) - F(\xb_{t_{s+1}}) &\geq \sum_{i=m_{s}+1}^{t_{s+1}-1}\frac{\|\xb_{i+1} - \xb_i\|_2^2}{4\etah} - 4\log(4/\delta)(t_{s+1}-m_s-1)\etah \frac{\sigma^2}{B} \notag \\
    &= (t_{s+1}-m_s-1)\bigg(\frac{\barD}{4\etah} - \frac{4\log(4/\delta)\etah\sigma^2}{B}\bigg)\notag\\
    &= (t_{s+1}-m_s-1)\bigg(\frac{\sigma^{2}}{16\etah BL^{2}} - \frac{4\log(4/\delta)\etah\sigma^2}{B}\bigg)\notag\\
    &\geq (t_{s+1}-m_s-1)\frac{4\log(4/\delta)\etah\sigma^2}{B}\label{eq:Afternoise},
\end{align}
where the last inequality is by the choice of $\etah \leq 1/ \big(L\sqrt{128\log(4/\delta)}\big)$.
For $i = m_s$, we have (See Lemma~\ref{lm:basic})
\begin{align}
F(\xb_{m_{s}+1}) &\leq F(\xb_{m_{s}}) + (\|\db_{m_{s}}\|_2 + \|\bepsilon_{m_{s}}\|_{2}+ Lr/2)r.\label{eq:basic3}
\end{align}
Plugging $\|\db_{m_{s}}\|_2 \leq \epsilon$ and $\|\bepsilon_{m_{s}}\|_2 \leq \epsilon/2$ into \eqref{eq:basic3} gives,  
\begin{align}
F(\xb_{m_{s}+1}) &\leq F(\xb_{m_{s}}) + (4\epsilon+ Lr/2)r\leq F(\xb_{m_{s}}) + \frac{2\log(4/\delta)\etah\sigma^2}{B}\label{eq:Atnoise},
\end{align}
where the last inequality is by the choice $r \leq \min\big\{\log(4/\delta)\etah\sigma^{2}/(2B\epsilon), \sqrt{2\log(4/\delta)\etah\sigma^2/(BL)}\big\}$.
% Adding noise may increase the function value by at most $2\log(4/\delta)\etah\sigma^2/B$ due to \eqref{eq:Atnoise}. After that, the function value will  on average decrease by at least $4\log(4/\delta)\etah\sigma^2/B$ due to \eqref{eq:Afternoise}. For the worst case, function value will only decrease once in $\ncstep$. 
Combining \eqref{eq:Afternoise} and \eqref{eq:Atnoise} and applying $t_{s+1}-m_s \geq 2$ gives,
\begin{align*}
F(\xb_{m_{s}}) - F(\xb_{t_{s+1}}) \geq  [4(t_{s+1}-m_s-1)-2]\frac{\log(4/\delta)\etah\sigma^2}{B} \geq (t_{s+1}-m_s)\frac{\log(4/\delta)\etah\sigma^2}{B}. 
\end{align*}
\end{proof}
Now, we can provide the proof of Theorem~\ref{Theorem Spider} .

\begin{proof}[Proof of Theorem~\ref{Theorem Spider}]
The analysis can be divided into two phases, i.e., $\gdphase$ and $\ncphase$.
The function value will decrease at different rates in different phases. 

\noindent\textbf{$\gdphase$}: In this phase, $\|\db_{t}\|_{2} \geq \epsilon$ and $\|\bepsilon\|_{2}\leq \epsilon/2$ due to Lemma~\ref{lm:boundeps}. Thus the gradients of the function are large $\|\nabla F(\xb)\|_{2} \geq \epsilon/2$. Lemma~\ref{lm:gradientdescent} further shows that the loss decreases by at least $ \sigma\epsilon/(16\sqrt{B}L)$ on average.

\noindent\textbf{$\ncphase$:} In this phase, the starting point $\xb_{m_{s}}$ satisfies $\|\nabla F(\xb_{m_{s}})\|_{2}\leq \|\db_{m_{s}}\|_{2}+\|\bepsilon_{t}\|_{2}\leq 2\epsilon$. If $\xb_{m_{s}}$ is a saddle point with 
$\lambda_{\min}(\nabla^2 F(\xb_{m_s})) \leq -\eh$, then by Lemma~\ref{lm:hessiandescent}, with high probability $\algspider$ will break $\ncphase$, set \textsc{find}$\leftarrow$False and begin a new $\gdphase$. Further by Lemma~\ref{lm:localization}, the loss will decrease by at least $\log(4/\delta)\etah\sigma^2/B$ per step on average.

\noindent\textbf{Sample Complexity:} Note that the total amount for function value can decrease is at most $\Delta = F(\xb_0) - \inf_\xb F(\xb) < +\infty$. So the algorithm must end and find an $(\epsilon, \eh)$-approximate local minimum within $\tilde O(\sqrt{B}L\Delta\sigma^{-1} \epsilon^{-1} + BL\Delta\sigma^{-2})$ iterations. Notice that on average we sample $\max\{b, B/q\} = \sqrt{B}$ examples per iteration,
so the total sample complexity is $\tilde O(BL\Delta\sigma^{-1} \epsilon^{-1} + B^{3/2}L\Delta\sigma^{-2})$.
% Notice that 
% \begin{align*}
% \frac{B^{3/2}\Delta}{\sigma^{2}} + \frac{B^{3/2}\Delta}{\sigma^{2}} + \frac{L^{3}\sigma}{\epsilon^{3}} \geq 3\frac{BL\Delta}{\sigma \epsilon}     
% \end{align*}
Plugging in the choice of $B = \tilde{O}(\sigma^{2}\epsilon^{-2} + \sigma^{2}\rho^{2}\eh^{-4})$ in Theorem~\ref{Theorem Spider}, we have the total gradient complexity
\begin{align*}
\tilde O\bigg(\frac{\sigma L\Delta}{\epsilon^{3}} + \frac{\sigma \rho^{2}L\Delta}{\epsilon\epsilon_{H}^{4}} + \frac{\sigma\rho^{3}L\Delta}{\epsilon_{H}^{6}}\bigg) = \tilde O\bigg(\frac{\sigma L\Delta}{\epsilon^{3}} + \frac{\sigma\rho^{3}L\Delta}{\epsilon_{H}^{6}}\bigg),  
\end{align*}
where the equation is due to the Young's inequality. 
% The proof finishes by using Young's inequality.

\end{proof}

\section{Conclusions and Future Work}
In this paper, we propose a perturbed stochastic gradient framework named $\algname$ for finding local minima. %To be specific, 
$\algname$ %together with  SARAH/SPIDER and STORM 
can find $(\eg,\eh)$-approximate local minima within $\tilde O(\eg^{-3}+\eh^{-6})$ stochastic gradient evaluations, which matches the best possible complexity results in the classical $\eh = \sqrt{\eg}$ setting. Our results show that simple perturbed gradient methods can be as efficient as more sophisticated algorithms for finding local minima in the classical setting. Recall that there still exists a mismatch between our upper bound $(\epsilon^{-3}+ \epsilon_H^{-6})$ and the lower bound $(\epsilon^{-3}+ \epsilon_H^{-5})$ in the general case of $\epsilon_H$, and we leave it as a future work to close this gap. %(
%Given remark 5.2 I still think there is no need to add it.)}

% Acknowledgments---Will not appear in anonymized version
\acks{We thank the anonymous reviewers for their helpful comments. ZC, DZ and QG are partially supported by the National Science Foundation IIS-2008981, CAREER Award 1906169, and AWS Machine Learning Research Award. The views and conclusions contained in this paper are those of the authors and should not be interpreted as representing any funding agencies.}

\bibliography{LENA}

\appendix

\section{Experiments}\label{sec:exp}
In  this  section,  we  conduct  some   experiments  to  validate  our  theory. We consider the symmetric matrix sensing problem. We need to recover a low-rank matrix $\Mb^{*} = \Ub^{*}(\Ub^{*})^{\top}$, where $\Ub^{*} \in \mathbb{R}^{d\times r}$. We have $n$ sensing matrix $\{\Ab_{i}\}_{i\in [n]}$ with observation $\bb_{i} = \langle \Ab_{i}, \Mb^{*}\rangle$. The optimization problem can be written as 

$$\min_{\Ub \in \mathbb{R}^{d\times r}}f(\Ub) = \frac{1}{2n}\sum_{i=1}^{n}(\langle \Ab_{i}, \Ub\Ub^{\top}\rangle -\bb_i)^{2}.$$
For the data generation, we consider $d=50, r=3$ and $d=100, r=3$. Then we generate the unknown low-rank matrix $\Mb^{*} = \Ub^{*}(\Ub^{*})^{\top}$, where every element in $\Ub^{*} \in \mathbb{R}^{d\times r}$ is independently drawn from the Gaussian distribution $\cN(0, 1/d)$. We then generate $n = 20d$ random sensing matrices  $\{\Ab_{i}\}_{i\in [n]}$ following standard normal distribution, and thus $\bb_{i} = \langle \Ab_{i}, \Mb^{*}\rangle$. The global optimal value of the above optimization problem is $0$, because there is no noise in the model. Next we randomly initialize a vector $\tilde{\ub}_{0}\in \RR^{d}$ from the Gaussian distribution. Then we set $\ub_{0} = \alpha \tilde{\ub}_{0}$ where $\alpha$ is a small constant to guarantee that $\|\ub_{0}\|_{2}< \lambda_{\max}(\Mb^{*}) $ and set the initial input $\Ub_{0} = [\ub_{0}, \zero, \zero]$. We have fixed initialization $\Ub_{0}$ for every optimization algorithm.

We choose our algorithm as $\algspider$ and take SGD, perturbed SGD \citep{ge2015escaping}, SPIDER \citep{fang2018spider}, $\text{SPIDER-SFO}^{+}$(+Neon2) \citep{fang2018spider} and SSRGD \citep{li2019ssrgd} as the baseline algorithms to compare. We evaluate the performance by objective function $\|\Ub\Ub^{\top} - \Mb^{*}\|_{F}^{2}/\|\Mb^{*}\|_{F}^{2}$ and then report the objective function value versus the number of stochastic gradient evaluations in Figure \ref{fig1}. We can see that without adding noise or using second-order information, SGD and SPIDER are not able to escape from saddle points (i.e., the objective function value of the converged point is far above zero). Our algorithm ($\algname$-SPIDER), SSRGD, Perturbed SGD and $\text{SPIDER-SFO}^{+}$(+Neon2) can escape from saddle points. Compared with SSRGD and perturbed SGD, our algorithm converges to the unknown matrix faster. 

\begin{figure}[H]
\vskip -0.1in
     \centering
     \subfigure[Matrix Sensing ($d=50$)]{\includegraphics[width=0.48\textwidth]{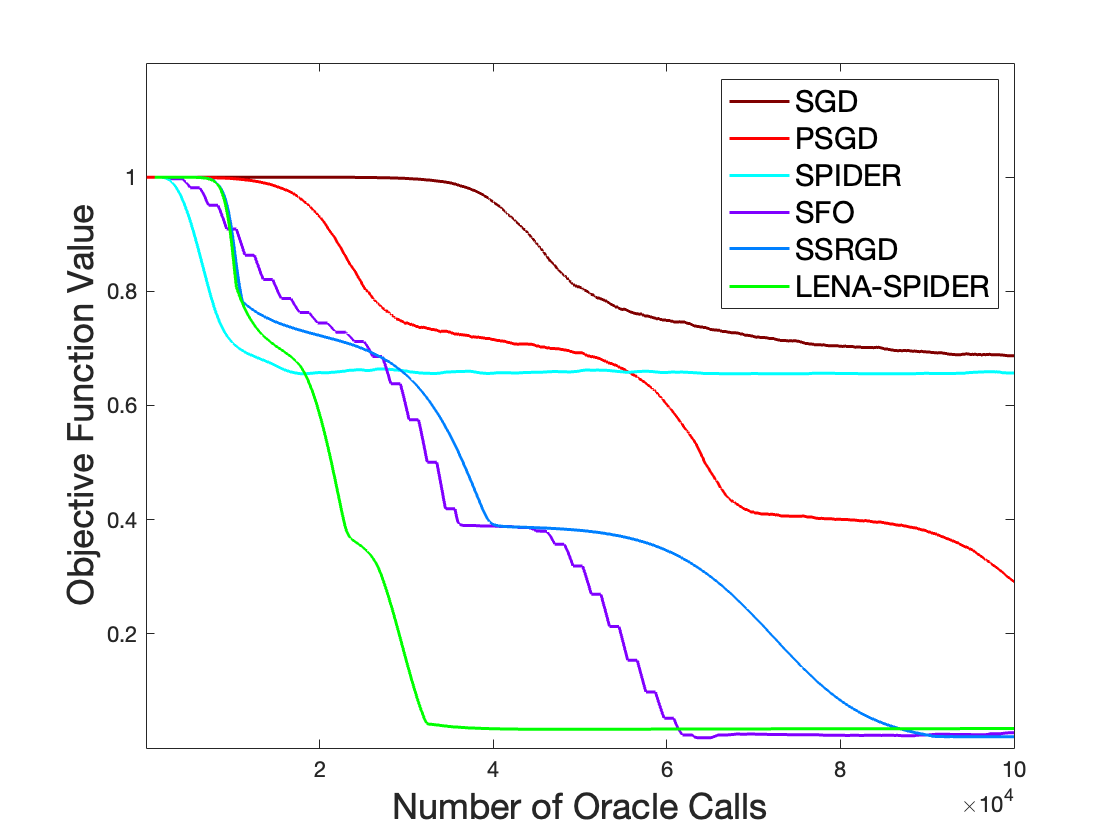}}
      \subfigure[Matrix Sensing ($d=100$)]{\includegraphics[width=0.48\textwidth]{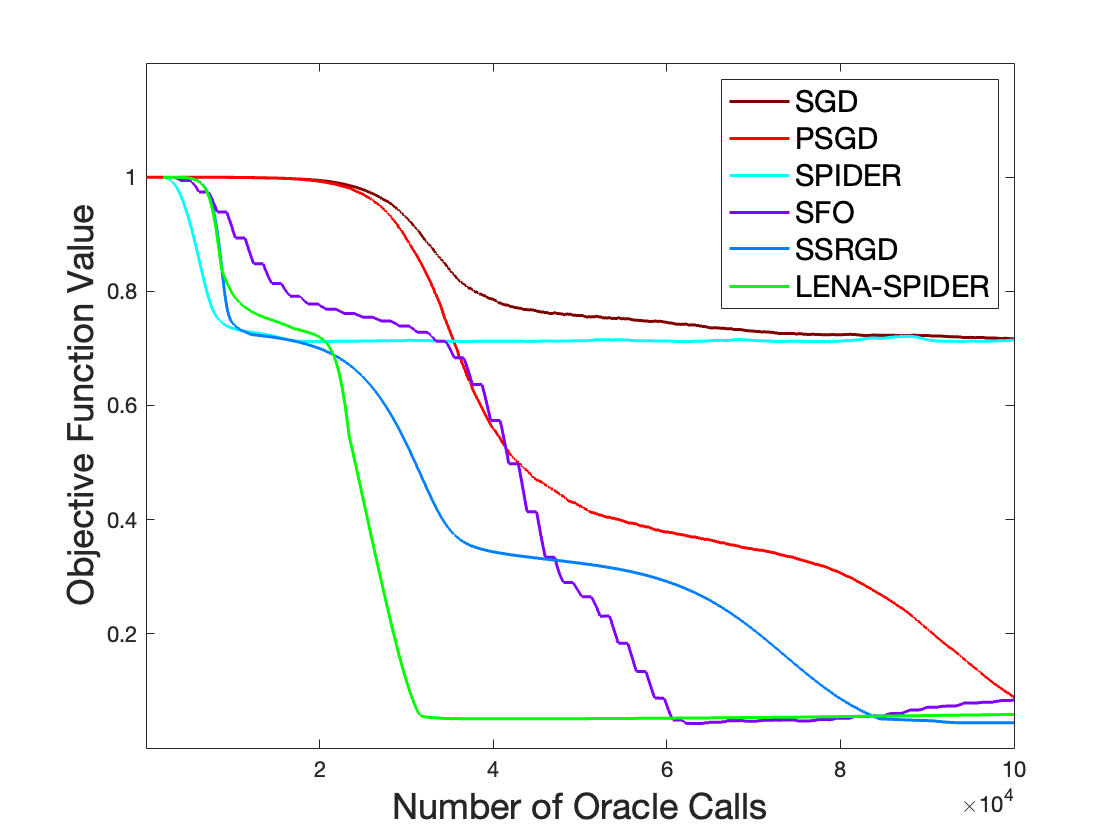}}
      \vskip -0.1in
    \caption{Convergence of different
algorithms for matrix sensing: objective function value versus the number of oracle calls}
    \label{fig1}
    \vskip -0.15in
\end{figure}

Our algorithm empirically outperforms the NEON2-based algorithm $\text{SPIDER-SFO}^{+}$, which can be seen through the experiment results. The reason is that the accuracy of the negative curvature estimation is very crucial to the success of NEON2-based algorithms. However, we found that the accuracy heavily depends on the number of iterations in the NEON2 algorithm, which requires careful parameter tuning to balance the computational cost and the accuracy. In contrast, our algorithm only relies on gradient descent-type updates besides an added noise, which is easier to tune.

\section{Proof of Theorem \ref{thm:storm} }\label{sec:mainproof}
In this section we present the main proof to Theorem \ref{thm:storm}.  We define $\bepsilon_t = \db_t - \nabla F(\xb_t)$ for simplicity.

To prove the main theorem, we need two groups of lemmas to charctrize the behavior of the Algorithm $\algstorm$. 

Next lemma provides the upper bound of $\bepsilon_{t}$.
\begin{lemma}\label{lm:boundeps2}
Set $\etag \leq \sigma/(2bL)$, $r \leq \sigma/(2bL)$ and $\barD \leq \sigma^{2}/(4b^{2}L^{2})$,  $a = 56^{2}\log(4/\delta)/b$, $B=b^{2}$,$a \leq 1/4\ttres$ , with probability at least $1-2\delta$, for all $t$ we have 
\begin{align*}
\|\bepsilon_{t}\|_2 \leq \frac{2^{10}\log(4/\delta)\sigma}{b}.    
\end{align*}
Furthermore, by the choice of $b$ in Theorem \ref{Theorem Spider} we have that $\|\bepsilon_{t}\|_{2} \leq \eg/2$.
\begin{proof}
See Appendix~\ref{proof:lm:boundeps2}.
\end{proof}
% \begin{align*}
% \|\bepsilon_{t}\|_2^2 &\leq \log(4 / \delta)\bigg[\frac{64L^2}{B}\sum_{i=0}^{t-1}(1 - a)^{2t - 2i}  \|\xb_{i+1} - \xb_i\|^2_2 + \frac{16a\sigma^2}{B}\bigg] + 2(1-a)^{2t-2s}\|\bepsilon_{0}\|_2^2.
% \end{align*}
% Specifically, for $0=0$ with probability $1-2\delta$, we have that
% \begin{align*}
%     \|\bepsilon_t\|_2^2 &\le \log(4 / \delta)\bigg[\frac{64L^2}{B}\sum_{i=0}^{t-1}(1 - a)^{2t - 2i}  \|\xb_{i+1} - \xb_i\|^2_2+ \frac{16a\sigma^2}{B} + \frac{32(1 - a)^{2t}\sigma^2}{S_0}\bigg].
% \end{align*}

\end{lemma}

\begin{lemma}\label{lm:gradientdescent2}
Suppose the event in Lemma \ref{lm:boundeps2} holds and $\eta \leq \epsilon/(2L)$, then for any $s$, we have
\begin{align}
    F(\xb_{t_s}) - F(\xb_{m_{s}}) \geq \frac{(m_{s} - t_{s})\eta\eg}{8}.\notag
\end{align}
\end{lemma}
\begin{proof}
The proof is the same as that of Lemma \ref{lm:gradientdescent}, with the fact $\|\bepsilon_t\|_2 \leq \epsilon/2$ from Lemma \ref{lm:boundeps2}. 
\end{proof}
The choice of $\eta$ in Theorem \ref{thm:storm} further implies that the loss decrease by $\sigma \epsilon/(16bL)$ on average.

% \begin{lemma}\label{lm:gradientdescent}
% Suppose 
% \begin{align}
%     \gmove \geq 2\epsilon, \gthres \geq 4\epsilon, \gthres/\gmove \geq \eta L,\notag 
% \end{align}
% then for any $s$, we have
% \begin{align}
%     F(\xb_{m_{s}}) - F(\xb_{t_s}) \leq -\eta(m_s - t_s)\epsilon^2.\notag
% \end{align}
% \end{lemma}

% \begin{lemma}\label{lm:epsgradient}
% For any $s$ and $t_s \leq t \leq m_s + 1$, we have
% \begin{align}
%     \|\bepsilon_t\|_2 \leq 10\epsilon
% \end{align}
% \end{lemma}
% Remark: choose r is small enough $r\leq L^{-1}\epsilon$

Next lemma shows that if $\xb_{m_{s}}$ is a saddle point, then with high probability, the algorithm will break during the $\ncphase$ and set \textsc{find}$\leftarrow$false. Thus, whenever $\xb_{m_{s}}$ is not a local minimum, the algorithm cannot terminate.

\begin{lemma}\label{lm:hessiandescent2}
Under Assumptions \ref{asm:lip} and \ref{asm:bdvariance}, set  perturbation radius $r\leq L\etah \eh/\rho$, $a\leq \etah \eh$, $b\geq \max\{16\log(4/\delta)\eta_{H}^{-2}L^{-2}\epsilon_{H}^{-2}, 56^{2}\log(4/\delta)a^{-1}\}$,$\ttres = 2\log(8\epsilon_{H}\sqrt{d}\rho^{-1}\delta^{-1}r^{-1})/(\eta_{H} \epsilon_{H})$, $\eta_{H} \leq \min\{1/(10L\log(8\epsilon_{H}L\rho^{-1}r_{0}^{-1})), 1/(10 L\log(\ttres))\}$ and $\barD < L^{2}\etah^{2}\eh^{2}/(\rho\ttres^{2})$. Then for any $s$, when $\lambda_{\min}(\nabla^2 F(\xb_{m_s})) \leq -\eh$,  with  probability at least $1-2\delta$ algorithm breaks in the $\ncphase$.
\end{lemma}
\begin{proof}
See Appendix \ref{proof:lm:hessiandescent2}. 
\end{proof}

% \begin{lemma}\label{lm:hessiandescent2}
% Suppose $B\geq \max\{64\log(4/\delta)\eta_{H}^{-2}L^{-2}\epsilon_{H}^{-2}, 56^{2}\log(4/\delta)a^{-1}\eta_{H}^{-2}L^{-2}\}$, $\eta_{H} \leq \tilde{O}\big(L^{-1}\big)$, $\ttres = \tilde{O}(\eta_{H}^{-1}\epsilon_{H}^{-1})$. $\barD = \tilde{O}\big(L^{2}\etah^{4}\eh^{4}\rho^{-2}\big)$, then for any $s$, when algorithm does not break, with high probability, $\xb_{m_s}$ satisfies $\lambda_{\min}(\nabla^2 F(\xb_{m_s})) \geq -\eh$.
% \end{lemma}

% \begin{lemma}\label{lm:gradientdescent2}
% By the choice of parameter in Theorem \ref{Theorem Spider}, for any $s$, we have
% \begin{align}
%     F(\xb_{t_s}) - F(\xb_{m_{s}}) \geq \frac{(m_{s} - t_{s})\eta\eg}{8}.\notag
% \end{align}
% The choice of $\eta$ in Theorem \ref{Theorem Spider} further implies that the loss decrease by $\frac{\sigma \epsilon}{16BL}$ on average.
% \end{lemma}

Next lemma shows that $\algstorm$ decreases when it breaks.
\begin{lemma}[localization]\label{lm:localization2}
Suppose the event in Lemma \ref{lm:boundeps2} holds, set perturbation radius $r \leq \min\big\{\log(4/\delta)^{2}\etah\sigma^{2}/(4b^{2}\epsilon), \sqrt{2\log(4/\delta)^{2}\etah\sigma^2/(b^{2}L)}\big\}$, $\etah \leq 1/ \big(2^{12}L\log(4/\delta)\big)$, and $\barD = \sigma^{2}/(4b^{2}L^{2})$.
Then for any $s$, when $\algstorm$ breaks, then $\xb_{m_s}$ satisfies 
% \begin{align}
%     \sum_{i=m_{s}+1}^{t-1}\|\xb_{i+1} - \xb_i\|_2^2 \leq  \etah(F(\xb_{m_{s}+1}) - F(\xb_{t})) + \etah^2 (t-m_s-1)\eg^2.
% \end{align}
% and 
% \begin{align}
%     F(\xb_{m_{s}}) - F(\xb_{t_{s+1}}) \geq (t_{s+1}-m_s)\frac{\sigma^{2}}{32\etah BL^{2}}.
% \end{align}
\begin{align}
    F(\xb_{m_{s}}) - F(\xb_{t_{s+1}}) \geq (t_{s+1}-m_s)\frac{\log(4/\delta)^{2}\etah\sigma^2}{b^{2}}.
\end{align}

\end{lemma}
\begin{proof}
See Appendix \ref{proof:lm:localization2}. 
\end{proof}

With all above lemmas, we prove Theorem \ref{thm:storm}.
\begin{proof}[Proof of Theorem \ref{thm:storm}]
Under the choice of parameter in Theorem \ref{thm:storm}, we have Lemma \ref{lm:boundeps2} to \ref{lm:localization2} hold. Now for \gdphase, we know that the function value $F$ decreases by $\sigma \epsilon/(16bL)$ on average. For \ncphase, we know that the $F$ decreases by $\log(4/\delta)^{2}\etah\sigma^2/b^{2}$ on average. 
So $\algstorm$ can find $(\epsilon, \eh)$-approximate local minima within $\tilde O(bL\Delta\sigma^{-1} \epsilon^{-1} + b^{2}L\Delta\sigma^{-2})$ iterations (we use the fact that $\etah = \tilde O(L^{-1})$). Then the total number of stochastic gradient evaluations is bounded by $\tilde O(B+ b^{2}L\Delta\sigma^{-1} \epsilon^{-1} + b^{3}L\Delta\sigma^{-2})$.
% Notice that 
% \begin{align*}
% \frac{B^{3/2}\Delta}{\sigma^{2}} + \frac{B^{3/2}\Delta}{\sigma^{2}} + \frac{L^{3}\sigma}{\epsilon^{3}} \geq 3\frac{BL\Delta}{\sigma \epsilon}     
% \end{align*}
Plugging in the choice of $b = \tilde{O}(\sigma\epsilon^{-1} + \sigma\rho\eh^{-2})$ in Theorem~\ref{thm:storm}, we have the total sample complexity
\begin{align*}
\tilde O\bigg(\frac{\sigma L\Delta}{\epsilon^{3}} + \frac{\sigma \rho^{2}L\Delta}{\epsilon\epsilon_{H}^{4}} + \frac{\sigma\rho^{3}L\Delta}{\epsilon_{H}^{6}}\bigg).  
\end{align*}
The proof finishes by using Young's inequality.
\end{proof}

\section{Proof of Lemmas in Section \ref{sec:mainproof}}\label{sec:induction}
In this section we prove lemmas in Section \ref{sec:mainproof}. Let filtration $\cF_{t, b}$ denote the all history before sample $\bxi_{t, b}$ at time $t \in \{0, \cdots, T\}$, then it is obvious that $\cF_{0, 1} \subseteq \cF_{0, b} \subseteq \cdots \subseteq \cF_{1, 1} \subseteq \cdots \subseteq \cF_{T,1} \subseteq \cdots \subseteq \cF_{T, b}$.

We also need the following fact: 
\begin{proposition}\label{prop:equal}
For any $t$, we have the following equation:
\begin{align}
       \frac{\bepsilon_{t+1}}{(1 - a)^{t+1}} - \frac{\bepsilon_{t}}{(1 - a)^{t}} = \frac{1}{(1-a)^{t+1}}\sum_{i\leq b} \bepsilon_{t,i}, \notag
\end{align}
where
\begin{align}
    \bepsilon_{t,i} &= \frac{a}{b}[\nabla f(\xb_{t+1}; \bxi_{t+1}^{i}) - \nabla F(\xb_{t+1})]\notag \\
    &\qquad + \frac{1 - a}{b}\big[\nabla F(\xb_t) - \nabla f(\xb_t; \bxi_{t+1}^{i}) - \nabla F(\xb_{t+1}) + \nabla f(\xb_{t+1}; \bxi_{t+1}^{i})].\notag
\end{align}
\end{proposition}
\begin{proof}
Following the update rule in $\algstorm$, we could have the update rule of $\bepsilon$ described as
\begin{align*}
    \bepsilon_{t+1} &= \frac{1 - a}{b}\sum_{i \leq b}\big[\db_t - \nabla f(\xb_t; \bxi_{t + 1}^i)\big] + \frac{1}{b}\sum_{i \leq b}\big[\nabla f(\xb_{t+1}; \bxi_{t+1}^{i}) - \nabla F(\xb_{t+1})\big] \\
    &= \frac{a}{b}\sum_{i \leq b}[\nabla f(\xb_{t+1}; \bxi_{t+1}^{i}) - \nabla F(\xb_{t+1})] + (1 - a)(\db_t - \nabla F(\xb_t))\\
    &\quad + \frac{1 - a}{b}\sum_{i \leq b}\big[\nabla F(\xb_t) - \nabla f(\xb_t; \bxi_{t+1}^{i}) - \nabla F(\xb_{t+1}) + \nabla f(\xb_{t+1}; \bxi_{t+1}^{i})\big]\\
    &= \frac{a}{b}\sum_{i \leq b}[\nabla f(\xb_{t+1}; \bxi_{t+1}^{i}) - \nabla F(\xb_{t+1})] + (1 - a)\bepsilon_{t}\\
    &\quad + \frac{1 - a}{b}\sum_{i \leq b}\big[\nabla F(\xb_t) - \nabla f(\xb_t; \bxi_{t+1}^{i}) - \nabla F(\xb_{t+1}) + \nabla f(\xb_{t+1}; \bxi_{t+1}^{i})\big],
\end{align*}
where the last equation is by definition $\bepsilon_t:=\db_t - \nabla F(\xb_t)$. Thus we have 
\begin{align}
    &\frac{\bepsilon_{t+1}}{(1 - a)^{t+1}} - \frac{\bepsilon_{t}}{(1 - a)^{t}} \notag \\
    &= \frac{1}{(1-a)^{t+1}}\Big(\frac{a}{b}\sum_{i \leq b}[\nabla f(\xb_{t+1}; \bxi_{t+1}^{i}) - \nabla F(\xb_{t+1})] \notag \\
    &\quad+ \frac{1 - a}{b}\sum_{i \leq b}\big[\nabla F(\xb_t) - \nabla f(\xb_t; \bxi_{t+1}^{i}) - \nabla F(\xb_{t+1}) + \nabla f(\xb_{t+1}; \bxi_{t+1}^{i})\big]\Big),\notag \\
    & = \frac{1}{(1-a)^{t+1}}\sum_{i\leq b} \bepsilon_{t,i}.\notag
\end{align}
\end{proof}

\subsection{Proof of Lemma~\ref{lm:boundeps2}}\label{proof:lm:boundeps2}
\begin{proposition}\label{very basic}
For two positive sequences $\{a_{i}\}_{i=1}^{n}$ and $\{b_{i}\}_{i=1}^{n}$. Suppose $C = \max_{i,j\in[n]}\{|a_{i}/a_{j}|\}$, $\bar{b} = \sum_{i=1}^{n}b_{i}/n$. Then we have,
\begin{align*}
\sum_{i=1}^{n}a_{i}b_{i} \leq \max_i a_i\cdot n \cdot \bar b\leq C\sum_{i=1}^{n}a_{i}\bar{b}.    
\end{align*}
\end{proposition}

\begin{proof}[Proof of Lemma~\ref{lm:boundeps2}]
By Proposition \ref{prop:equal} we have
\begin{align}
    &\frac{\bepsilon_{t+1}}{(1 - a)^{t+1}} - \frac{\bepsilon_{t}}{(1 - a)^{t}} 
    = \frac{1}{(1-a)^{t+1}}\sum_{i\leq b} \bepsilon_{t,i}.\notag
\end{align}
It is easy to verify that $\{\bepsilon_{t,i}\}$ forms a martingale difference sequence and
\begin{align}
    \|\bepsilon_{t,i}\|_2^2 &\leq 2\bigg\|\frac{a}{b}[\nabla f(\xb_{t+1}; \bxi_{t+1}^{i}) - \nabla F(\xb_{t+1})] \bigg\|_2^2\notag\\
    &\quad+ 2\bigg\|\frac{1 - a}{b}\big[\nabla F(\xb_t) - \nabla f(\xb_t; \bxi_{t+1}^{i}) - \nabla F(\xb_{t+1}) + \nabla f(\xb_{t+1}; \bxi_{t+1}^{i})\big]\bigg\|_2^2\notag\\
    & \leq \frac{2a^2\sigma^2 + 8(1-a)^2L^2\|\xb_{t+1} - \xb_i\|_2^2}{b^2},\notag
\end{align}
where the first inequality holds due to triangle inequality, the second inequality holds due to Assumptions \ref{asm:lip} and \ref{asm:bdvariance}. Therefore, by Azuma-Hoeffding inequality (See Lemma \ref{lm:ah-vec} for detail), with probability at least $1-\delta$, we have that for any $t>0$, 
\begin{align}
    \bigg\|\frac{\bepsilon_{t}}{(1 - a)^{t}} - \frac{\bepsilon_{0}}{(1 - a)^{0}}\bigg\|_2^2 &\leq 4\log(4/\delta)\sum_{i=0}^{t-1}b\cdot \frac{2a^2\sigma^2 + 8(1-a)^2L^2\|\xb_{i+1} - \xb_i\|_2^2}{(1-a)^{2i+2}b^{2}}\notag \\
    & = 8\log(4/\delta)\sum_{i=0}^{t-1} \frac{a^2\sigma^2 + 4(1-a)^2L^2\|\xb_{i+1} - \xb_i\|_2^2}{(1-a)^{2i+2}b}\notag.
\end{align}
Therefore, we have 
\begin{align}
\|\bepsilon_{t}\|_2^2 &\leq  2(1-a)^{2t}\bigg\|\frac{\bepsilon_{t}}{(1 - a)^{t}} - \bepsilon_{0}\bigg\|_2^2 +  2(1-a)^{2t}\|\bepsilon_{0}\|_2^2\notag\\
&\le \log(4 / \delta)\bigg[\frac{64L^2}{b}\sum_{i=0}^{t-1}(1 - a)^{2t - 2i}  \|\xb_{i+1} - \xb_i\|^2_2 + \frac{16a\sigma^2}{b}\bigg] + 2(1-a)^{2t}\|\bepsilon_{0}\|_2^2\label{eq:epsilon1}.
\end{align}
By Azuma-Hoeffding Inequality, we have with probability $1-\delta$,
\begin{align*}
    \|\bepsilon_0\|_2^2 = \bigg\|\frac{1}{B}\sum_{1\leq i \leq B}\Big[\nabla f(\xb_{0}; \bxi_{0}^{i}) - \nabla F(\xb_{0})\Big]\bigg\|_2^2 \leq \frac{4\log(4/\delta)\sigma^2}{B}. 
\end{align*}
Therefore, with probability $1- 2\delta$, we have
\begin{align}
    \|\bepsilon_t\|_2^2 &\le \log(4 / \delta)\bigg[\frac{64L^2}{b}\sum_{i=0}^{t-1}(1 - a)^{2t - 2i-2}  \|\xb_{i+1} - \xb_i\|^2_2+ \frac{16a\sigma^2}{b} + \frac{32(1 - a)^{2t}\sigma^2}{B}\bigg]\notag \\
    &= \frac{64L^2\log(4/\delta)}{b}\underbrace{\sum_{i=0}^{t-1}(1 - a)^{2t - 2i-2}  \|\xb_{i+1} - \xb_i\|^2_2}_{I} + \frac{16a\sigma^2\log(4/\delta)}{b}\\
    &\qquad  + \frac{32(1-a)^{2t}\log(4/\delta)\sigma^2}{B}\label{eq:110}.
\end{align}
We now bound $I$. Denote $S_{1} = \{i \in [t-1]| \exists j, t_{j}\leq i < m_{j}\}$, $S_{2} = \{i \in [t-1]| \exists j, i = m_{j}\}$, $S_{3} = \{i \in [t-1]| \exists j, m_{j} < i < t_{j+1}\}$, We can divide $I$ into three part, 
\begin{align}
    I &= \underbrace{\sum_{i \in S_{1}}(1 - a)^{2t - 2i-2}  \|\xb_{i+1} - \xb_i\|^2_2}_{I_1} + \underbrace{\sum_{i\in S_{2}}(1 - a)^{2t - 2i-2}  \|\xb_{i+1} - \xb_i\|^2_2}_{I_2}\notag \\
    &\qquad + \underbrace{\sum_{i \in S_{3}}^{t-1} (1 - a)^{2t - 2i-2}  \|\xb_{i+1} - \xb_i\|^2_2}_{I_3}.\label{eq:111}
\end{align}
Because $\|\xb_{i+1} - \xb_{i}\|_{2} = \eta_{t}\|\db_{i}\|_{2} = \eta$, we can bound $I_{1}$ as follows,
\begin{align}
    I_{1} = \eta^2\sum_{i \in S_{1}}(1 - a)^{2t - 2i-2} \leq \eta^2\sum_{i=0}^{\infty}(1 - a)^{i} = \frac{\eta^{2}}{a}.\label{eq:112}
\end{align}
Because the perturbation radius is $r$, we can bound $I_{2}$ as follows, 
\begin{align}
    I_{2} = \sum_{i\in S_{2}}(1 - a)^{2t - 2i-2}  \|\xb_{i+1} - \xb_i\|^2_2 \leq r^{2}\sum_{i\in S_{2}}(1 - a)^{2t - 2i-2}\leq \frac{r^{2}}{a}.\label{eq:113}
\end{align}
To bound $I_3$, we have
\begin{align}
    I_3 &= \sum_{i \in S_{3}}^{t-1} (1 - a)^{2t - 2i-2}  \|\xb_{i+1} - \xb_i\|^2_2 \notag \\
    & =  \sum_{s=1}^S\sum_{i = m_{s}+1}^{\min\{t-1, t_{s+1}-1\}} (1 - a)^{2t - 2i-2}  \|\xb_{i+1} - \xb_i\|^2_2 \notag \\
    &\leq \sum_{s=1}^S(1-a)^{-2\ttres}\sum_{i = m_{s}+1}^{\min\{t-1, t_{s+1}-1\}}  (1 - a)^{2t - 2i-2} \barD \notag \\
    & = (1-a)^{-2\ttres}\sum_{i \in S_{3}}^{t-1}(1 - a)^{2t - 2i-2} \barD\notag \\
    &\leq \frac{\barD(1-a)^{-2\ttres}}{a},
\label{eq:117}
\end{align}
where $S$ satisfies $m_S < t-1 < t_{S+1}$. The first inequality holds due to Proposition \ref{very basic} with the fact that the average of $\|\xb_{i+1} - \xb_i\|^2_2$ is bounded by $\bar D$, according to the $\algname$ scheme, and $t_{s+1} - m_s<\ttres$, the last one holds trivially. Substituting \eqref{eq:112}, \eqref{eq:113}, \eqref{eq:117} into \eqref{eq:111}, we have
\begin{align}\nonumber
    I \leq\frac{\eta^{2}+r^{2}+(1-a)^{2\ttres}\barD}{a}.
\end{align}
Therefore \eqref{eq:110} can further bounded by 
\begin{align}
 \|\bepsilon_t\|_2^2 &\leq \frac{64L^2\log(4/\delta)}{b}\frac{\eta^{2}+r^{2}+(1-a)^{2\ttres}\barD}{a} + \frac{16a\sigma^2\log(4/\delta)}{b} + \frac{32(1-a)^{2t}\log(4/\delta)\sigma^2}{B}\label{eq:118}.    
\end{align}

By the selection of  $\etag \leq \sigma/(2bL)$, $r \leq \sigma/(2bL)$ and $\barD \leq \sigma^{2}/(4b^{2}L^{2})$,  $a = 56^{2}\log(4/\delta)/b$, $B=b^{2}$,$a \leq 1/4\ttres$, it's easy to verify that
\begin{align}
 &\frac{64L^2\log(4/\delta)}{b}\frac{\eta^{2}+r^{2}+ 2\barD}{a} \leq \frac{\sigma^2}{b^{2}}\label{eq:119}\\
  &(1-a)^{2\ttres} \geq 1 - 2a\ttres \geq \frac{1}{2}\label{eq:120}\\
   &\frac{16a\sigma^2\log(4/\delta)}{b} \leq \frac{224^{2}\sigma^2\log(4/\delta)^{2}}{b^{2}}\label{eq:121}\\
   &\frac{32\log(4/\delta)\sigma^2}{B} \leq  \frac{32\log(4/\delta)\sigma^2}{b^{2}}\label{eq:122}.
\end{align}
Plugging \eqref{eq:119} to \eqref{eq:122} into \eqref{eq:118} gives, 
\begin{align*}
\|\bepsilon_{t}\|_2 \leq \frac{2^{10}\log(4/\delta)\sigma}{b}.    
\end{align*}

\end{proof}

% \subsection{Proof of Lemma \ref{lm:gradientdescent}}
% We need the following lemma
% \begin{lemma}\label{lm:basic}
% For any $t \neq m_s$, we have
% \begin{align}
% F(\xb_{t+1}) \leq F(\xb_t) - \eta_t\|\db_t\|_2^2/2 + \eta_t\|\bepsilon_t\|_2^2/2 + \frac{L}{2}\|\xb_{t+1} - \xb_t\|_2^2.\notag
% \end{align}
% For $t = m_s$, we have
% \begin{align}
%      F(\xb_{t+1}) \leq F(\xb_t) + (\gthres + Lr/2+\|\bepsilon_t\|_2)r.\notag
% \end{align}
% \end{lemma}
% \begin{proof}[Proof of Lemma \ref{lm:gradientdescent}]
% Suppose $t_s \leq t <m_s$, then we have
% \begin{align}
%     F(\xb_{t+1}) &\leq F(\xb_t) - \eta_t\|\db_t\|_2^2/2 + \eta_t\|\bepsilon_t\|_2^2/2 + \frac{L}{2}\|\xb_{t+1} - \xb_t\|_2^2\notag \\
%     & =  F(\xb_t) -\|\xb_{t+1} - \xb_t\|_2^2\bigg(\frac{1}{2\eta_t} - \frac{L}{2}\bigg) + \eta_t\|\bepsilon_t\|_2^2/2\notag \\
%     & = F(\xb_t) -\eta^2\bigg(\frac{1}{2\eta_t} - \frac{L}{2}\bigg) + \eta_t\|\bepsilon_t\|_2^2/2,\notag
% \end{align}
% where the inequality holds due to Lemma \ref{lm:basic}. Note that $\eta_t \leq \eta\gmove/\gthres$, then we have
% \begin{align}
%     F(\xb_{t+1}) &\leq F(\xb_t) -\eta^2\gmove^2\bigg(\frac{\gthres}{\eta\gmove} - \frac{L}{2}\bigg) + 2\epsilon^2\eta\frac{\gmove}{\gthres} \leq -\eta \epsilon^2\notag
% \end{align}
% due to the choice of $\eta, \gmove, \gthres$. Taking summation from $t = t_s$ to $m_{s}-1$, that completes our proof. 
% \end{proof}

\subsection{Proof of Lemma \ref{lm:hessiandescent2}}\label{proof:lm:hessiandescent2}

\begin{lemma}[Small stuck region]\label{lemma:small stuck2}
Suppose $-\gamma = \lambda_{\min}(\nabla^2 F(\xb_{m_s})) \leq -\epsilon_{H}$. 
Set threshold $ \ell  =  2\log(8\epsilon_{H}\rho^{-1}r_{0}^{-1})/(\eta_{H} \gamma)$, $\eta_{H} \leq \min\{1/(10L\log(8\epsilon_{H}L\rho^{-1}r_{0}^{-1})), 1/(10 L\log(\ell))\}$, $a \leq \etah\gamma $, $r \leq L\etah\eh/\rho$.
Let $\{\xb_{t}\}, \{\xb_{t}'\}$ be two coupled sequences by running $\algstorm$ from $\xb_{m_{s}+1}, \xb_{m_{s}+1}'$ with $\wb_{m_{s}+1} = \xb_{m_{s}+1}-\xb_{m_{s}+1}' = r_{0}\eb_{1}$, where $\xb_{m_{s}+1}, \xb_{m_{s}+1}'\in \mathbb{B}_{\xb_{m_{s}}}(r)$, $r_{0} = \delta r/\sqrt{d}$ and $\eb_{1}$ denotes the smallest eigenvector direction of Hessian $\nabla^{2}F(\xb_{m_{s}})$. Moreover, let batch size $b\geq \max\{16\log(4/\delta)\eta_{H}^{-2}L^{-2}\gamma^{-2}, 56^{2}\log(4/\delta)a^{-1}\}$, then with probability $1-2\delta$ we have 
\begin{align*}
\exists T\leq \ell, \max\{\|\xb_{T} - \xb_{0}\|_{2},\|\xb_{T}' - \xb_{0}'\|_{2}\} \geq \frac{\eta_{H}\epsilon_{H}L}{\rho}.   
\end{align*}
\end{lemma}
\begin{proof}
See Appendix \ref{proof:lemma:small stuck2}. 
\end{proof}

\begin{proof}[Proof of Lemma~\ref{lm:hessiandescent2}]
We assume $\lambda_{\min}(\nabla^2 F(\xb_{m_s})) <-\eh$ and prove our statement by contradiction. 
Lemma \ref{lemma:small stuck2} shows that, in the random perturbation ball at least one of two points in the $\eb_{1}$ direction will escape the saddle point if their distance is larger than $r_{0} = \frac{\delta r}{\sqrt{d}}$. Thus, the probability of the starting point $\xb_{m_{s}+1} \sim \mathbb{B}_{\xb_{m_{s}}}(r)$ located in the stuck region uniformly is less than $\delta$. Then with probability at least $1 - 2\delta$, 
\begin{align}
\exists m_s < t <m_s +\ttres,  \|\xb_t - \xb_{m_s}\|_2  \geq \frac{L\etah\eh}{\rho}. \label{help1}
\end{align}
Suppose $\algstorm$ does not break, then for any $m_s < t <m_s +\ttres$,
\begin{align}
\|\xb_t - \xb_{m_s}\|_2  \leq \sum_{i=m_s}^{t-1}\|\xb_{i+1} - \xb_i\|_2 \leq \sqrt{(t-m_s)\sum_{i=m_s}^{t-1} \|\xb_{i+1} - \xb_i\|_2^2} \leq (t-m_s)\sqrt{\barD},\notag
\end{align}
where the first inequality is due to the triangle inequality and the second inequality is due to Cauchy-Schwarz inequality. Thus, by the selection of $\barD$, we have
\begin{align}
    \|\xb_t - \xb_{m_s}\|_2 \leq(t-m_s)\sqrt{\barD} \leq \ttres \sqrt{\barD} < \frac{L\etah\eh}{\rho},\notag
\end{align}
which contradicts \eqref{help1}. Therefore, we know that with probability at least $1-2\delta$, $\lambda_{\min}(\nabla^2 F(\xb_{m_s})) \geq -\eh$. 
\end{proof}

\subsection{Proof of Lemma \ref{lm:localization2}}\label{proof:lm:localization2}

\begin{proof}[Proof of Lemma \ref{lm:localization2}]

Suppose $m_s < i <t_{s+1}$. Then with probability at least $1-\delta$, then by Lemma \ref{lm:basic} we have
\begin{align}
F(\xb_{i+1}) 
&\leq F(\xb_i)+ \frac{\eta_i}{2}\|\bepsilon_i\|_2^2 - \bigg(\frac{1}{2\eta_i} - \frac{L}{2}\bigg)\|\xb_{i+1} - \xb_i\|_2^2 \notag \\
&\leq F(\xb_i)+ \frac{\etah}{2}\frac{2^{20}\log(4/\delta)^{2}\sigma^2}{b^2} - \frac{1}{4\etah}\|\xb_{i+1} - \xb_i\|_2^2\label{eq:local1}
\end{align}
where the the second inequality holds due to Lemma \ref{lm:boundeps2} and the fact that for any $m_s < i < t_{s+1}$, $\eta_i \leq \etah \leq 1/(2L)$. Taking summation of $\eqref{eq:local1}$ from $i=m_{s}+1$ to $t-1$, we have
\begin{align}
F(\xb_{t}) &\leq F(\xb_{m_{s}+1}) + 2^{19}\etah \log(4/\delta)^{2}(t-m_s-1)\frac{\sigma^2}{b^2} - \frac{1}{4\etah}\sum_{i=m_{s}+1}^{t-1}\|\xb_{i+1} - \xb_i\|_2^2.
\end{align}
Finally, we have
\begin{align}
    F(\xb_{m_{s}+1}) - F(\xb_{t_{s+1}}) &\geq \sum_{i=m_{s}+1}^{t_{s+1}-1}\frac{\|\xb_{i+1} - \xb_i\|_2^2}{4\etah} - 2^{19}\log(4/\delta)^{2}(t-m_s-1)\etah \frac{\sigma^2}{b^2} \notag \\
    &= (t_{s+1}-m_s-1)\bigg(\frac{\barD}{4\etah} - \frac{2^{19}\log(4/\delta)^{2}\etah\sigma^2}{b^2}\bigg)\notag\\
    &= (t_{s+1}-m_s-1)\bigg(\frac{\sigma^{2}}{16\etah b^2L^{2}} - \frac{2^{19}\log(4/\delta)^{2}\etah\sigma^2}{b^2}\bigg)\notag\\
    &\geq (t_{s+1}-m_s-1)\frac{4\log(4/\delta)^{2}\etah\sigma^2}{b^2}\label{eq:Afternoise2},
\end{align}
where the last inequality is by the selection of $\etah \leq 1/ \big(2^{12}L\log(4/\delta)\big)$.
For $i = m_s$, by Lemma \ref{lm:basic} we have
\begin{align}
     F(\xb_{m_{s}+1}) &\leq F(\xb_t) + (2\|\db_t\|_2 + 2\|\bepsilon_{t}\|_{2}+ Lr/2)r\notag\\
     &\leq F(\xb_{m_{s}}) + (4\epsilon+ Lr/2)r\notag\\
     &\leq F(\xb_{m_{s}}) + \frac{2\log(4/\delta)^{2}\etah\sigma^2}{b^2}\label{eq:Atnoise2},
\end{align}
where the last inequality is by $r \leq \min\big\{\log(4/\delta)^{2}\etah\sigma^{2}/(4b^2\epsilon), \sqrt{2\log(4/\delta)^{2}\etah\sigma^2/(b^2L)}\big\}$.

Combining \eqref{eq:Afternoise2} and \eqref{eq:Atnoise2} we have that 
\begin{align*}
F(\xb_{m_{s}}) - F(\xb_{t_{s+1}}) & = F(\xb_{m_{s}}) - F(\xb_{m_{s}+1}) + F(\xb_{m_{s}+1}) - F(\xb_{t_{s+1}})\notag \\
& \geq(t_{s+1}-m_s-1) \frac{4\log(4/\delta)^{2}\etah\sigma^2}{b^2} - \frac{2\log(4/\delta)^{2}\etah\sigma^2}{b^2}\notag \\
& \geq (t_{s+1}-m_s)\frac{\log(4/\delta)^{2}\etah\sigma^2}{b^2},   
\end{align*}
where we use the fact that $t_{s+1} - m_s \geq 2$. 
\end{proof}

\section{Proof of Lemmas in Section \ref{sec:induction}}

\subsection{Proof of Lemma~\ref{lemma:small stuck2}}\label{proof:lemma:small stuck2}

Define $\wb_{t} := \xb_{t} - \xb_{t}'$ as the distance between the two coupled sequences. By the construction, we have that $\wb_{0} = r_{0}\eb_{1}$, where $\eb_{1}$ is the smallest eigenvector direction of Hessian $\cH:=\nabla^2F(\xb_{m_{s}})$. 
\begin{align*}
\wb_{t} &= \wb_{t-1} - \eta(\db_{t-1} - \db_{t-1}')\\
&= \wb_{t-1} - \eta(\nabla F(\xb_{t-1}) - \nabla F(\xb_{t-1}') + \db_{t-1} -F(\xb_{t-1}) - \db_{t-1}' + \nabla F(\xb_{t-1}'))\\
&= \wb_{t-1} - \eta\bigg[(\xb_{t-1}-\xb_{t-1}')\int_{0}^{1}\nabla^{2}F(\xb_{t-1}' + \theta(\xb_{t-1}-\xb_{t-1}'))d\theta \\
&\qquad + \db_{t-1} -F(\xb_{t-1}) - \db_{t-1}' + F(\xb_{t-1}')\bigg]\\
&= (1-\eta \cH)\wb_{t-1} - \eta(\Delta_{t-1}\wb_{t-1} + \yb_{t-1}),
\end{align*}
where 
\begin{align}
    &\Delta_{t-1} := \int_{0}^{1}\big(\nabla^{2}F(\xb_{t-1}' + \theta(\xb_{t-1}-\xb_{t-1}')) - \cH\big)d\theta,\notag \\
    &\yb_{t-1} := \db_{t-1} - \nabla F(\xb_{t-1})-\db_{t-1}' + \nabla F(\xb_{t-1}') = \bepsilon_{t-1} - \bepsilon_{t-1}'.\notag
\end{align}
Recursively applying the above equation, we get
\begin{align}
\wb_{t}=(1-\eta \cH)^{t-m_{s}-1}\wb_{m_{s}+1} - \eta\sum_{\tau=m_{s}+1}^{t-1}(1-\eta\cH)^{t-1-\tau}(\Delta_{\tau}\wb_{\tau} + \yb_{\tau})\label{eq:small1}.       
\end{align}
We want to show that the first term of \eqref{eq:small1} dominates the second term. Next Lemma is essential for the proof of Lemma~\ref{lemma:small stuck2}, which bounds the norm of $\yb_{t}$.

\begin{lemma}\label{Lemma:ybound}
Under Assumption~\ref{asm:lip}, we have following inequality holds,
\begin{align}
\|\yb_{t}\|_2 &\leq 2\sqrt{\log(4/\delta)}b^{-1/2}a^{-1/2}\Big(2L\max_{m_{s}<\tau<t}\|\wb_{\tau+1} - \wb_{\tau}\|_{2}\notag\\
&\quad + \max_{m_{s}<\tau \leq t}(2aL + 4\rho D_{\tau})\cdot \max_{m_{s}<\tau\leq t}\|\wb_{\tau}\|_{2}\Big) +  4\sqrt{\log(4/\delta)}b^{-1/2}Lr_{0} \label{eq:small2}, 
\end{align}
where $D_{\tau} = \max\{\|\xb_{\tau} - \xb_{m_{s}}\|_{2}, \|\xb_{\tau}' - \xb_{m_{s}}\|_{2}\}$.
\end{lemma}
\begin{proof}[Proof of Lemma~\ref{Lemma:ybound}]
By Proposition \ref{prop:equal}, we have that
\begin{align*}
\frac{\yb_{t+1}}{(1-a)^{t+1}} - \frac{\yb_t}{(1-a)^{t}} &= \frac{\bepsilon_{t+1}}{(1 - a)^{t+1}} - \frac{\bepsilon_{t}}{(1 - a)^{t}} - \frac{\bepsilon_{t+1}'}{(1 - a)^{t+1}} + \frac{\bepsilon_{t}'}{(1 - a)^{t}} \notag \\
&= \frac{1}{(1-a)^{t+1}}\sum_{i\leq b}[\bepsilon_{t,i} - \bepsilon_{t,i}'],
\end{align*}
where $\bepsilon_{t,i}$ is the same as that in Proposition \ref{prop:equal}:
\begin{align}
    \bepsilon_{t,i} &= \frac{a}{b}[\nabla f(\xb_{t+1}; \bxi_{t+1}^{i}) - \nabla F(\xb_{t+1})] \notag\\
    &\qquad + \frac{1 - a}{b}\big[\nabla F(\xb_t) - \nabla f(\xb_t; \bxi_{t+1}^{i}) - \nabla F(\xb_{t+1}) + \nabla f(\xb_{t+1}; \bxi_{t+1}^{i})]\notag\\
    &= \frac{1}{b}[\nabla f(\xb_{t+1}; \bxi_{t+1}^{i}) - \nabla F(\xb_{t+1})] + \frac{1 - a}{b}\big[\nabla F(\xb_t) - \nabla f(\xb_t; \bxi_{t+1}^{i})],\label{eq:newview}
\end{align}
where we rewrite $\bepsilon_{t,i}$ as \eqref{eq:newview} because now we want bound the $\bepsilon_{t} - \bepsilon_{t}'$ by the distance between two sequence. $\bepsilon_{t,i}'$ is defined similarly as follows 
\begin{align*}
    \bepsilon_{t,i}' &=\frac{1}{b}[\nabla f(\xb_{t+1}; \bxi_{t+1}^{i}) - \nabla F(\xb_{t+1})] + \frac{1 - a}{b}\big[\nabla F(\xb_t) - \nabla f(\xb_t; \bxi_{t+1}^{i})].
\end{align*}
It is easy to verify that $\{\bepsilon_{t,i} - \bepsilon_{t,i}'\}$ forms a martingale difference sequence. We now bound $\|\bepsilon_{t,i} - \bepsilon_{t,i'}\|_{2}^{2}$. Denote $\cH_{t+1,i} =\nabla^2f(\xb_{m_{s}}; \bxi_{t+1}^{i})$, then we introduce two terms 
\begin{align*}
&\Delta_{t+1, i} := \int_{0}^{1}\big(\nabla^{2}f(\xb_{t+1}' + \theta(\xb_{t+1}-\xb_{t+1}'); \bxi_{t+1}^{i}) - \cH_{t+1,i}\big)d\theta\\
&\hat{\Delta}_{t+1, i} := \int_{0}^{1}\big(\nabla^{2}f(\xb_{t}' + \theta(\xb_{t}-\xb_{t}'); \bxi_{t+1}^{i}) - \cH_{t+1,i}\big)d\theta, 
\end{align*}
By Assumption~\ref{asm:lip}, we have $\|\Delta_{t+1,i}\|_{2}\leq \rho\max_{\theta\in[0,1]}\|\xb_{t+1}' + \theta(\xb_{t+1}-\xb_{t+1}') - \xb_{m_{s}+1}\|_{2} \leq \rho D_{t+1}$, similarly we have $\|\hat{\Delta}_{t+1, i}\|_{2} \leq \rho D_{t}$ and $\Delta_{t+1} \leq \rho D_{t+1}$.

Now we bound $\bepsilon_{t,i} - \bepsilon_{t,i}'$, 
\begin{align}
b(\bepsilon_{t,i} - \bepsilon_{t,i}')
&= \Big([\nabla f(\xb_{t+1}; \bxi_{t+1}^{i}) - \nabla F(\xb_{t+1})] + (1 - a)\big[\nabla F(\xb_t) - \nabla f(\xb_t; \bxi_{t+1}^{i})\big]\Big)\notag\\
&\quad- \Big([\nabla f(\xb_{t+1}'; \bxi_{t+1}^{i}) - \nabla F(\xb_{t+1}')] - (1-a)\big[\nabla F(\xb_{t}') - \nabla f(\xb_{t}'; \bxi_{t+1}^{i})\big]\Big)\notag\\
&= \big(\cH_{t+1,i}\wb_{t+1}+ \Delta_{t+1,i}\wb_{t+1} - \cH\wb_{t+1} - \Delta_{t+1}\wb_{t+1} + (1-a)\cH\wb_{t}\notag\\
&\quad  + (1-a)\Delta_{t}\wb_{t} -  (1-a)\cH_{t+1,i}\wb_{t}  - (1-a)\hat{\Delta}_{t+1,i}\wb_{t} \big)\notag\\
&= \big(\cH_{t+1,i} - \cH\big)\big(\wb_{t+1} - (1-a)\wb_{t}\big) + \big(\Delta_{t+1,i}-\Delta_{t+1}\big)\wb_{t+1}\notag\\
&\qquad + (1-a)\big(\Delta_{t} -  \hat{\Delta}_{t+1,i})\wb_{t}\label{eq:sma1l3}.
\end{align}
This implies the LHS of \eqref{eq:sma1l3} has the following bound.
\begin{align*}
\|b(\bepsilon_{t,i} - \bepsilon_{t,i}')\|_{2} &\leq 2L\|\wb_{t+1} - (1-a)\wb_{t}\|_{2} + 2\rho D_{t+1}^{x}\|\wb_{t+1}\|_{2} + 2\rho D_{t}^{x}\|\wb_{t}\|_{2}\\
&\leq 2L\|\wb_{t+1} - \wb_{t}\|_{2} + 2\rho D_{t+1}^{x}\|\wb_{t+1}\|_{2} + (2aL + 2\rho D_{t}^{x})\|\wb_{t}\|_{2}\\
&\leq \underbrace{2L\max_{m_{s}<\tau<t}\|\wb_{\tau+1} - \wb_{\tau}\|_{2} + \max_{m_{s}<\tau \leq t}(2aL + 4\rho D_{\tau})\cdot \max_{m_{s}<\tau\leq t}\|\wb_{\tau}\|_{2}}_{M}
\end{align*}
where the first inequality is by the gradient Lipschitz Assumption and Hessian Lipschitz Assumption~\ref{asm:lip}, the second inequality is by triangle inequality. Therefore we have 
\begin{align*}
\|\bepsilon_{t,i} - \bepsilon_{t,i}' \|_{2}^{2} \leq \frac{M^{2}}{b^{2}}     
\end{align*}
Furthermore, by Azuma Hoeffding inequality(See Lemma \ref{lm:ah-vec} for detail), with probability at least $1-\delta$, we have that for any $t>0$, 
\begin{align*}
    \bigg\|\frac{\yb_{t}}{(1 - a)^{t}} - \frac{\yb_{m_{s}+1}}{(1-a)^{m_{s}+1}}\bigg\|_2^2 &=  \bigg\|\sum_{\tau=m_{s}+1}^{t-1}\bigg(\frac{\yb_{\tau+1}}{(1 - a)^{\tau+1}} - \frac{\yb_{\tau}}{(1 - a)^{\tau}} \bigg) \bigg\|_2^2 \\ 
    &= \bigg\|\sum_{\tau=m_{s}+1}^{t-1}\bigg(\frac{1}{(1-a)^{\tau+1}}\sum_{i\leq b}[\bepsilon_{\tau,i} - \bepsilon_{\tau,i}'] \bigg) \bigg\|_2^2 \\ 
    &\le 4\log(4 / \delta) \bigg(\sum_{i=m_{s}+1}^{t-1}b\cdot \frac{M^{2}}{(1 - a)^{2\tau + 2}b^{2}}\bigg).
\end{align*}
Multiply $(1-a)^{2t}$ on both side, we get
\begin{align*}
    \|\yb_{t} - (1-a)^{t-m_{s}-1}\yb_{m_{s}+1}\|_2^2 
    &\leq 4b^{-1}\log(4 / \delta) \sum_{\tau=m_{s}+1}^{t-1}(1 - a)^{2t-2\tau - 2}M^{2} \notag\\
    &\leq 4\log(4 / \delta)b^{-1}a^{-1}M^{2}, 
\end{align*}
where the last inequality is by $\sum_{i=0}^{t-1}(1 - a)^{2t-2i - 2} \leq a^{-1}$. Furthermore, by triangle inequality we have

\begin{align}
\|\yb_{t}\|_{2} \leq  2\sqrt{\log(4 / \delta)}b^{-1/2}a^{-1/2}M + (1-a)^{t-m_{s}-1}\|\yb_{m_{s}+1}\|_{2}. \label{eq: ybound1}   
\end{align}
Next, we have $\|\nabla f(\xb_{m_{s}+1}; \bxi_{m_{s}+1}^{i}) - \nabla F(\xb_{m_{s}+1}') - \nabla f(\xb_{m_{s}+1}'; \bxi_{m_{s}+1}^{i}) + \nabla F(\xb_{m_{s}+1}')\|_{2}\leq 2Lr_{0}$ due to Assumption~\ref{asm:lip}. Then by Azuma Inequality (See Lemma \ref{lm:ah-vec}), we have with probability at least $1-\delta$,
\begin{align}
\|\yb_{m_{s}+1}\|_{2}^{2} &= \|\db_{m_{s}+1} - \nabla F(\xb_{m_{s}+1})-\db_{m_{s}+1}' + \nabla F(\xb_{m_{s}+1}')\|_{2}^{2}\notag\\
&= \bigg\|\frac{1}{b}\sum_{i\leq b}[\nabla f(\xb_{m_{s}+1}; \bxi_{m_{s}+1}^{i}) - \nabla F(\xb_{m_{s}+1}') - \nabla f(\xb_{m_{s}+1}'; \bxi_{m_{s}+1}^{i}) + \nabla F(\xb_{m_{s}+1}')]\bigg\|_{2}^{2}\notag\\
&\leq \frac{4\log(4/\delta)4L^{2}r_{0}^{2}}{b}\label{eq:ybound2}.
\end{align}
Plugging \eqref{eq:ybound2} into \eqref{eq: ybound1} gives
\begin{align*}
\|\yb_{t}\|_2 &\leq 2\sqrt{\log(4/\delta)}b^{-1/2}a^{-1/2}\Big(2L\max_{m_{s}<\tau<t}\|\wb_{\tau+1} - \wb_{\tau}\|_{2}\notag\\
&\quad + \max_{m_{s}<\tau \leq t}(2aL + 4\rho D_{\tau})\cdot \max_{m_{s}<\tau\leq t}\|\wb_{\tau}\|_{2}\Big) +  4\sqrt{\log(4/\delta)}b^{-1/2}Lr_{0}.  
\end{align*}
\end{proof}

Now we can give a proof of Lemma~\ref{lemma:small stuck2}.
\begin{proof}[Proof of Lemma~\ref{lemma:small stuck2}]
We proof it by induction that
\begin{enumerate}
\item $\frac{1}{2}(1+\eta_{H} \gamma)^{t-m_{s}-1}r_{0} \leq  \|\wb_{t}\|_{2} \leq \frac{3}{2}(1+\eta_{H}\gamma)^{t-m_{s}-1}r_{0}.$
\item $\|y_{t}\|_{2} \leq 2\eta_{H} \gamma L(1+\eta_{H} \gamma)^{t-m_{s}-1}r_{0}$.
\end{enumerate}
First for $t=m_{s}+1$, we have $\|\wb_{m_{s}+1}\|_{2} = r_{0}$, $\|y_{m_{s}+1}\|_{2} \leq \sqrt{16b^{-1}\log(4/\delta)L^2r_{0}^2} \leq 2\eta_{H}\gamma L r_{0}$(See \eqref{eq:ybound2}), where $b \geq  2\eta_{H}^{-2}\gamma^{-2}\sqrt{\log(4/\delta)}$.
Assume they hold for all $m_{s} < \tau <t$, we now prove they hold for t. We bound $\wb_{t}$ first, we only need to show that second term of \eqref{eq:small1} is bounded by $\frac{1}{2}(1+\eta_{H}\gamma)^{t}r_0$.
\begin{align*}
&\bigg\|\eta_{H}\sum_{\tau=m_{s}+1}^{t-1}(1-\eta_{H}\cH)^{t-1-\tau}(\Delta_{\tau}\wb_{\tau} + \yb_{\tau})\bigg\|_{2}\notag \\
&\leq \eta_{H}\sum_{\tau=m_{s}+1}^{t-1}(1+\eta_{H}\gamma)^{t-1-\tau}(\|\Delta_{\tau}\|_{2}\|\wb_{\tau}\|_{2} + \|\yb_{\tau}\|_{2})\\
&\leq \eta_{H}\sum_{\tau=m_{s}+1}^{t-1}(1+\eta_{H}\gamma)^{t-m_{s}-2}r_{0}(\frac{3}{2}\|\Delta_{\tau}\|_{2}+ 2\eta_{H} \gamma L)\\
&\leq  \eta_{H}\sum_{\tau=m_{s}+1}^{t-1}(1+\eta_{H}\gamma)^{t-m_{s}-2}r_{0}(3\eta_{H}\epsilon_{H}L+ 2\eta_{H} \gamma L)\\
&= \eta_{H} \ell(1+\eta_{H}\gamma)^{t-m_{s}-2}r_{0}\cdot 5\eta_{H} \gamma L\\
&\leq 10\log(8 \eh \rho^{-1}r_{0}^{-1})\etah L (1+\eta_{H}\gamma)^{t-m_{s}-2}r_{0}\\
&\leq \frac{1}{2}(1+\eta_{H} \gamma)^{t-m_{s}-1}r_{0},
\end{align*}
where the first inequality is by the eigenvalue assumption over $\cH$, the second inequality is by the Induction hypothesis, the third inequality is by $\|\Delta_{\tau}\|_{2} \leq \rho D_{\tau} = \rho \max\{\|\xb_{\tau} - \xb_{m_{s}}\|_{2},\|\xb_{\tau}' - \xb_{m_{s}}\|_{2}\} \leq \etah\epsilon_{H}L + r\rho \leq 2\etah\epsilon_{H}L$, the fourth inequality is by the choice of $t - m_{s} -1 \leq \ell \leq 2\log(8\eh \rho^{-1}r_{0}^{-1})/(\etah\gamma)$, the last inequality is by the choice of $\etah\leq 1/(10\log(8\eh \rho^{-1}r_{0}^{-1})L)$.
Now we bound $\|\yb_{t}\|_{2}$ by \eqref{eq:small2}. We first get the bound for $L\|\wb_{i+1} - \wb_{i}\|_{2}$ as follows,
\begin{align}
&L\|\wb_{t+1} - \wb_{t}\|_{2} \notag\\ 
&= L\bigg\|-\etah\cH(I-\etah \cH)^{t- m_{s} - 2}\wb_{0} - \etah\sum_{\tau=m_{s}+1}^{t-2}\etah \cH(I-\etah \cH)^{t-2-\tau}(\Delta_{\tau}\wb_{\tau}+\yb_{\tau})  \notag \\
&\qquad +\etah(\Delta_{t-1}\wb_{t-1} + \yb_{t-1})\bigg\|_{2}\notag\\
&\overset{(i)}{\leq} L\etah \gamma(1+\etah \gamma)^{t-m_{s}-2}r_{0} + L\etah\bigg\|\sum_{\tau=m_{s}+1}^{t-2}\etah \cH(I-\etah \cH)^{t-2-\tau}(\Delta_{\tau}\wb_{\tau}+\yb_{\tau})\bigg\|_{2} \notag \\
&\qquad + L\etah\bigg\|\Delta_{t-1}\wb_{t-1} + \yb_{t-1}\bigg\|_{2}\notag\\
&\overset{(ii)}{\leq} L\etah \gamma(1+\etah \gamma)^{t-m_{s}-2}r_{0} \notag \\
&\qquad + L\etah\bigg[\bigg\|\sum_{\tau=m_{s}+1}^{t-2}\etah \cH(I-\etah \cH)^{t-2-\tau}\bigg\|_{2}+1\bigg]\max_{0\leq \tau \leq t-1}\bigg\|\Delta_{\tau}\wb_{\tau}+\yb_{\tau}\bigg\|_{2}\notag\\
&\overset{(iii)}{\leq}L\etah \gamma(1+\etah \gamma)^{t-m_{s}-2}r_{0} + L\etah\bigg[\sum_{\tau=m_{s}+1}^{t-2}\frac{1}{t-1-\tau}+1\bigg]\max_{0\leq \tau \leq t-1}\bigg\|\Delta_{\tau}\wb_{\tau}+\yb_{\tau}\bigg\|_{2}\notag\\
&\overset{(iv)}{\leq} L\etah \gamma(1+\etah \gamma)^{t-m_{s}-2}r_{0} +  L\etah[\log(t-m_{s}-1)+1] \cdot [5\etah \gamma L(1+\etah \gamma)^{t-m_{s}-2}r_{0}] \notag\\
&\overset{(v)}{\leq} 6L\etah \gamma(1+\etah \gamma)^{t-m_{s}-2}r_{0}+5\log(t-m_{s}-1) \gamma\eta_{H}^{2} L^{2}(1+\eta_{H} \gamma)^{t-m_{s}-2}r_{0}\label{eq:small4},
\end{align}
where (i) is by triangle inequality, (ii) is by the definition of max, (iii) is by $\|\etah \cH(I-\etah \cH)^{t-2-\tau}\|_{2} \leq \frac{1}{t-1-\tau}$, (iv) is due to $\|\Delta_{\tau}\|_{2} \leq \rho D_{\tau} \leq \rho (\eta_{H}\gamma L/\rho + r) \leq 2\gamma \eta_{H}L$, $\|\wb_{\tau}\|_{2}\leq 3(1+\eta_{H}\gamma)^{\tau-m_{s}-1}r_{0}/2$ and $\|\yb_{\tau}\|_{2} \leq 2\eta_{H} \gamma L(1+\eta_{H} \gamma)^{\tau-m_{s}-1}r_{0}$, (v) is due to $\etah \leq 1/L$.

We next get the bound of $ \max_{m_{s}<\tau \leq t}(2aL + 4\rho D_{\tau})\cdot \max_{m_{s}<\tau\leq t}\|\wb_{\tau}\|_{2}$ as follows
\begin{align}
\max_{m_{s}<\tau \leq t}(2aL + 4\rho D_{\tau})\cdot \max_{m_{s}<\tau\leq t}\|\wb_{\tau}\|_{2} &\leq (2aL + 8\gamma\eta_{H}L)\frac{3(1+\eta_{H} \gamma)^{t-m_{s}-1}}{2}r_{0}\notag\\
&\leq  15\gamma\eta_{H}L(1+\eta_{H} \gamma)^{t-m_{s}-1}r_{0}\label{eq:small5}.
\end{align}
where the first inequality is by $\rho D_{t} \leq \rho (\gamma\eta_{H}L/\rho + r) \leq 2\gamma\eta_{H}L$ and the induction hypothesis, last inequality is by $a \leq  \gamma\eta_{H}$.

Plugging \eqref{eq:small4} and \eqref{eq:small5}  into \eqref{eq:small2} gives,
\begin{align*}
\|\yb_{t}\|_2 &\leq 2\sqrt{\log(4/\delta)}b^{-1/2}a^{-1/2}\Big(2L\max_{m_{s}<\tau<t}\|\wb_{\tau+1} - \wb_{\tau}\|_{2}\notag\\
&\quad + \max_{m_{s}<\tau \leq t}(2aL + 4\rho D_{\tau})\cdot \max_{m_{s}<\tau\leq t}\|\wb_{\tau}\|_{2}\Big) +  4\sqrt{\log(4/\delta)}b^{-1/2}Lr_{0}\\
&\leq 2\sqrt{\log(4/\delta)}b^{-1/2}a^{-1/2}\Big(10\log(\ell) \gamma\eta_{H}^{2} L^{2} (1+\eta_{H} \gamma)^{t-m_{s}-1}r_{0}\\
&\quad +  27\gamma \etah L(1+\eta_{H} \gamma)^{t-m_{s}-1}r_{0} \Big) + 4\sqrt{\log(4/\delta)}b^{-1/2}Lr_{0}\\
&\leq  \underbrace{56\sqrt{\log(4/\delta)}b^{-1/2}a^{-1/2}\etah L \gamma (1+\eta_{H} \gamma)^{t-m_{s}-1}r_{0}}_{I_{1}} \notag \\
&\qquad + \underbrace{4\sqrt{\log(4/\delta)}b^{-1/2}(1+\eta_{H} \gamma)^{t-m_{s}-1}r_{0}}_{I_{2}}
\end{align*}
where the last inequality is by $\eta_{H} \leq 1/(10 L\log\ell)$. Now we bound $I_{1}$ and $I_{2}$ respectively. 
\begin{align*}
I_{1} &= 56\sqrt{\log(4/\delta)}b^{-1/2}a^{-1/2}\etah L\gamma (1+\eta_{H} \gamma)^{t-m_{s}-1}r_{0}\\
&= \eta_{H} \gamma L (1+\eta_{H} \gamma)^{t-m_{s}-1}r_{0},
\end{align*}
where the inequality is applying $b \geq  56^{2}\log(4/\delta)a^{-1}$.
Now we bound $I_{2}$ by applying $b\geq 16\log(4/\delta)\eta_{H}^{-2}L^{-2}\gamma^{-2}$,
\begin{align*}
I_{2} \leq \eta_{H} \gamma  L (1+\eta_{H} \gamma)^{t-m_{s}-1}r_{0}.   
\end{align*}
Then we obtain that
\begin{align*}
\|\yb_{t}\|_{2} \leq 2 \eta_{H} \gamma L (1+\eta_{H} \gamma)^{t-m_{s}-1}r_{0},
\end{align*}
which finishes the induction. So we have $\|\wb_{t}\|_{2}\geq \frac{1}{2}(1+\eta_{H} \gamma)^{t-m_{s}-1}r_{0}$. However, the triangle inequality give the bound
\begin{align*}
\|\wb_{t}\|_{2}&\leq \|\xb_{t}-\xb_{m_{s+1}}\|_{2} + \|\xb_{m_{s+1}}-\xb_{m_{s}}\|_{2} + \|\xb_{t}'-\xb_{m_{s}+1}'\|_{2} + \|\xb_{m_{s}+1}'-\xb_{m_{s}}'\|_{2}\\
&\leq 2r + 2\frac{\epsilon_{H}\eta_{H}L}{\rho}\\
&\leq 4\frac{\epsilon_{H}\eta_{H}L}{\rho},    
\end{align*}
where the last inequality is due to $r\leq \epsilon_{H}\eta_{H}L/\rho$. So we obtain that
\begin{align*}
t \leq \frac{\log(8\epsilon_{H}\eta_{H}L\rho^{-1}r_{0}^{-1})}{\log(1+\eta_{H} \gamma)}< \frac{2\log(8\epsilon_{H}\rho^{-1}r_{0}^{-1})}{\eta_{H} \gamma}.
\end{align*}
\end{proof}

\section{Auxiliary Lemmas}

We start by providing the Azuma–Hoeffding inequality under the vector settings.  
 
\begin{lemma}[Theorem 3.5, \citealt{pinelis1994optimum}]\label{lm:ah-vec}
Let $\bepsilon_{1:k} \in \RR^d$ be a vector-valued martingale
difference sequence with respect to $\cF_k$, i.e., for each $k \in [K]$, $\EE[\bepsilon_k|\cF_k] = 0$ and $\|\bepsilon_k\|_2 \le B_k$, then we have given $\delta \in (0, 1)$, w.p. $1 - \delta$, 
\begin{align*}
    \bigg\|\sum_{i=1}^K\bepsilon_k\bigg\|_2^2 \le 4\log(4 / \delta)\sum_{i=1}^KB_k^2.
\end{align*}
\end{lemma}
This lemma provides a dimension-free bound due to the fact that the Euclidean norm version of $\RR^{d}$ is $(2,1)$ smooth, see also \citet{kallenberg1991some, fang2018spider}. 

% We have the following lemma:

\begin{lemma}\label{lm:basic}
For any $t \neq m_s$, we have
\begin{align}
F(\xb_{t+1}) \leq F(\xb_t) - \frac{\eta_t}{2}\|\db_t\|_2^2 + \frac{\eta_t}{2}\|\bepsilon_t\|_2^2 + \frac{L}{2}\|\xb_{t+1} - \xb_t\|_2^2.\notag
\end{align}
For $t = m_s$, we have $ F(\xb_{t+1}) \leq F(\xb_t) + (\|\db_t\|_2 + \|\bepsilon_{t}\|_{2}+ Lr/2)r.$
\end{lemma}

\begin{proof}[Proof of Lemma~\ref{lm:basic}]
By Assumption \ref{asm:lip}, we have
\begin{align}
    F(\xb_{t+1}) \leq F(\xb_t) + \la \nabla F(\xb_t), \xb_{t+1} - \xb_t\ra + \frac{L}{2} \|\xb_{t+1} - \xb_t\|_2^2.
\end{align}
For the case $t \neq m_s$, the update rule is $\xb_{t+1} = \xb_t - \eta_t\db_t$, therefore 
\begin{align*}
    F(\xb_{t+1}) &\le F(\xb_t) - \eta_t\la \nabla F(\xb_t), \db_t\ra + \frac{L}{2} \|\xb_{t+1} - \xb_t\|_2^2\\
    &= F(\xb_t) - \eta_t\|\nabla F(\xb_t)\|_2^2/2 - \eta_t\|\db_t\|_2^2/2 + \eta_t\|\bepsilon_t\|_2^2/2 + L\|\xb_{t+1} - \xb_t\|_2^2/2\\
    &\le  F(\xb_t) - \eta_t\|\db_t\|_2^2/2 + \eta_t\|\bepsilon_t\|_2^2/2 + \frac{L}{2}\|\xb_{t+1} - \xb_t\|_2^2,
\end{align*}
where the first inequality on the first line is due to Assumption \ref{asm:lip} and the second inequality holds trivially. For the case $t = m_s$, since $\| \nabla F(\xb_t)\|_{2} \leq \|\db_t\|_2 + \|\bepsilon_{t}\|_{2}$ we have
\begin{align}
    F(\xb_{t+1}) &\leq F(\xb_t) + \la \nabla F(\xb_t), \xb_{t+1} - \xb_{t}\ra + \frac{L}{2} \|\xb_{t+1} - \xb_t\|_2^2 \notag \\
    & \leq F(\xb_t) + (\|\db_t\|_2 + \|\bepsilon_{t}\|_{2}+ Lr/2)r.\notag
\end{align}
\end{proof}

We present the following lemma from \citet{li2019ssrgd}, which characterizes the moving distance for a SPIDER-type stochastic gradient estimator during the Escape phase. Note that this lemma can be directly applied to our Algorithm~\ref{alg:Pullback} without modification since our algorithm and the SSRGD algorithm in \citet{li2019ssrgd} share the same Escape phase (See lines 8-14 in Algorithm~\ref{alg:Pullback} and lines 9-17 in SSRGD, Algorithm 2, \citealt{li2019ssrgd}).

% perturbation error of a SPIDER-type stochastic gradient estimator 
% \CC{The lines 8-14 of Algorithm~\ref{alg:Pullback} are the same as SSRGD \citep{li2019ssrgd} other than the last shrinkage step in line 12, so we have the following Lemma.}
\begin{lemma}[Lemma 6, \citealt{li2019ssrgd}]\label{lm:hessiandescentbaisc}
Suppose $-\gamma = \lambda_{\min}(\nabla^2 F(\xb_{m_s})) \leq -\epsilon_{H}$. 
Set perturbation radius $r \leq L\etah\eh/(C\rho)$, threshold $\ttres  = 2\log(\etah\epsilon_{H}\sqrt{d} LC^{-1}\rho^{-1}\delta^{-1}r^{-1})/(\etah\eh) = \tilde O(\etah^{-1}\eh^{-1})$,  step size $\etah\leq \min\{1/(16L\log(\etah\epsilon_{H}\sqrt{d} LC^{-1}\rho^{-1}\delta^{-1}r^{-1})), 1/(8CL\log \ttres) \} = \tilde O(L^{-1})$, $b = q = \sqrt{B} \geq 16\log(4/\delta)/(\etah^{2}\eh^{2})$. Let $\{\xb_{t}\}, \{\xb_{t}'\}$ be two coupled sequences by running $\algspider$ \ from $\xb_{m_{s}+1}, \xb_{m_{s}+1}'$ with $\wb_{m_{s}+1} = \xb_{m_{s}+1}-\xb_{m_{s}+1}' = r_{0}\eb_{1}$, where $\xb_{m_{s}+1}, \xb_{m_{s}+1}'\in \mathbb{B}_{\xb_{m_{s}}}(r)$, $r_{0} = \delta r/\sqrt{d}$ and $\eb_{1}$ denotes the smallest eigenvector direction of Hessian $\nabla^{2}F(\xb_{m_{s}})$. Then with probability at least $1-\delta$, 
\begin{align}
    \max_{m_s < t <m_s +\ttres}\{\|\xb_t - \xb_{m_s}\|_2, \|\xb_0 - \xb_{m_s}\|_2\} \geq \frac{L\etah\eh}{C\rho},
\end{align}
where $C = O(\log (d\ttres/\delta) = \tilde O(1)$.
\end{lemma}
% \begin{remark}
% This lemma can be simplified without the parameter $C$, because can use $\etah$ to eat the $\log$ term introduced by the length $\ttres$. This lemma doesn't have any dependency of $B$ because the lemma in \citet{li2019ssrgd} is for the full gradient, Lemma 6 has dependency of $B$, however $\eta \leq \log(\ttres)$, $\ttres \leq \log(\eta)$ makes the proof very strange. Thus I prefer to write a new one, based on the sample proof in \algstorm.
% \end{remark}

\end{document}